%% file: CombinatorialRandomKnots.tex
\documentclass[12pt]{amsart}
\usepackage{amsthm, amsrefs, amsmath, amssymb}
\usepackage[version-1-compatibility]{siunitx}
\usepackage{graphicx}
\usepackage{caption}
\usepackage{wrapfig}
\usepackage{tikz}
\usetikzlibrary{decorations.pathreplacing}

\newtheorem{theorem}{Theorem}[section]

\newtheorem{example}[theorem]{Example}

\newtheorem{definition}[theorem]{Definition}
\newtheorem{corollary}[theorem]{Corollary}
\newtheorem{conjecture}[theorem]{Conjecture}
\newtheorem{lemma}[theorem]{Lemma}

\newtheorem{consequence}[theorem]{Consequence}

\numberwithin{equation}{section}

\graphicspath{ {./} }

\begin{document}

\title{Combinatorial Random Knots}

\author{Andrew Ducharme, Emily Peters}
\thanks{Supported in part by NSF Grant DMS-1501116}
\date{\today}

\begin{abstract}
We explore free knot diagrams, which are projections of knots into the plane which don't record over/under data at crossings.  We  consider the combinatorial question of which free knot diagrams give which  knots and with what probability.  Every free knot diagram is proven to produce trefoil knots, and certain simple families of free knots are completely worked out.  We make some conjectures (supported by computer-generated data) about bounds on the probability of a knot arising from a fixed free diagram being the unknot, or being the trefoil.
\end{abstract}

\maketitle

\section{Introduction}
\input{Intro.tex}

\section{Background}\label{background}
\input{Background.tex}

\section{General results}\label{results}
\input{Results.tex}

\section{Tangles and the $n$-foil knots}
\input{Observations.tex}


\newpage

\section{Appendix}

The following tables lists the resultant knot probability for the unknot, trefoil, and figure eight, as well as a given resultant's expected number of minimum crossings for all knots with 9 or fewer crossings. Notice knots 8$_{17}$, 8$_{19}$, 8$_{20}$, and 8$_{21}$ have identical values. Since they share the same shape in the Rolfsen tabulation with differences arising in the choice of crossings they are the same in this context.

\begin{table}[h]
\begin{tabular}{c|c|c|c|c}
Knot   & Unknot \% & Trefoil \% & Figure Eight \% & Expectation Value \\
Unknot & 100       & 0          & 0               & 0                 \\
3$_1^{*\dagger}$   & 75        & 25         & 0               & 0.75              \\
4$_1 ^{\dagger \mathsection}$   & 75        & 12.5       & 12.5            & 0.875             \\
5$_1^*$  & 62.5      & 31.25      & 0               & 1.25              \\
5$_2 ^\dagger$   & 68.75     & 18.75      & 6.25            & 1.125             \\
6$_1 ^\dagger$   & 68.75     & 12.5       & 12.5            & 1.21875           \\
6$_2 ^\mathsection$   & 59.375    & 21.875     & 12.5            & 1.53125           \\
6$_3$   & 56.25     & 28.125     & 3.125           & 1.625             \\
7$_1^*$ & 54.6875   & 32.8125    & 0               & 1.640625          \\
7$_2 ^\dagger$   & 65.625    & 15.625     & 7.8125          & 1.375             \\
7$_3 ^\ddagger$   & 56.25     & 23.4375    & 6.25            & 1.703125          \\
7$_4$   & 60.9375   & 15.625     & 9.375           & 1.609375          \\
7$_5$   & 54.6875   & 26.5625    & 4.6875          & 1.75              \\
7$_6$   & 56.25     & 21.875     & 9.375           & 1.734375          \\
7$_7$   & 56.25     & 15.625     & 15.625          & 1.8125
\end{tabular}
\end{table}

$*=$ foil

$\dagger = 2 \, n$ 

$\ddagger = k \, n$ 

$\mathsection$ = 2 1 $n$

\newpage

\begin{table}[h]
\begin{tabular}{c|c|c|c|c}
Knot   & Unknot \% & Trefoil \% & Figure Eight \% & Expectation Value \\
8$_1 ^\dagger$   & 65.625    & 11.71875   & 11.71875        & 1.453125          \\
8$_2 ^\mathsection$   & 50.78125  & 24.21875   & 7.8125          & 1.9921875         \\
8$_3 ^\ddagger$   & 60.9375   & 12.5       & 12.5            & 1.6796875         \\
8$_4$   & 54.6875   & 18.75      & 10.15625        & 1.6796875         \\
8$_5$   & 46.875    & 25.78125   & 7.03125         & 2.2421875         \\
8$_6$   & 53.90625  & 21.875     & 9.375           & 1.8828125         \\
8$_7$   & 47.65625  & 21.09375   & 3.125           & 2.1171875         \\
8$_8$   & 52.34375  & 25         & 5.46875         & 1.9375            \\
8$_9$   & 48.4375   & 25.78125   & 7.03125         & 2.1171875         \\
8$_{10}$  & 45.3125   & 30.46875   & 2.34375         & 2.28125           \\
8$_{11}$  & 52.34375  & 17.96875   & 11.71875        & 2.03125           \\
8$_{12}$  & 56.25     & 15.625     & 15.625          & 1.84375           \\
8$_{13}$  & 50.78125  & 21.875     & 7.03125         & 2.0859375         \\
8$_{14}$  & 53.125    & 18.75      & 11.71875        & 1.984375          \\
8$_{15}$  & 48.4375   & 25         & 6.25            & 2.2109375         \\
8$_{16}$  & 41.40625  & 26.5625    & 3.90625         & 2.609375          \\
8$_{17}$  & 44.53125  & 23.4375    & 7.03125         & 2.4921875         \\
8$_{18}$  & 34.375    & 28.125     & 1.5625          & 3.03125           \\
8$_{19}$  & 44.53125  & 23.4375    & 7.03125         & 2.4921875         \\
8$_{20}$  & 44.53125  & 23.4375    & 7.03125         & 2.4921875         \\
8$_{21}$  & 44.53125  & 23.4375    & 7.03125         & 2.4921875        
\end{tabular}
\end{table}

$*=$ foil

$\dagger = 2 \, n$ 

$\ddagger = k \, n$ 

$\mathsection$ = 2 1 $n$

\newpage

\begin{table}[h]
\begin{tabular}{c|c|c|c|c}
Knot & Unknot \% & Trefoil \% & Figure Eight \% & Expectation Value \\
 9$_1^*$ & 49.2188 & 32.8125 & 0. & 1.96875 \\
 9$_2^\dagger$ & 63.6719 & 13.6719 & 8.20313 & 1.57031 \\
   9$_3^\ddagger$ & 48.8281 & 24.6094 & 5.85938 & 2.12891 \\
   9$_4^\ddagger$ & 53.125 & 19.5313 & 7.8125 & 2.00781 \\
   9$_5$ & 57.0313 & 13.6719 & 9.76563 & 1.92969 \\
   9$_6$ & 46.875 & 28.125 & 3.90625 & 2.19531 \\
   9$_7$ & 51.5625 & 24.2188 & 5.85938 & 2.02344 \\
   9$_8$ & 52.3438 & 21.4844 & 8.98438 & 2.02344 \\
   9$_9$ & 46.0938 & 26.9531 & 4.6875 & 2.26953 \\
   9$_{10}$ & 49.2188 & 18.75 & 9.375 & 2.23438 \\
   9$_{11}$ & 47.2656 & 24.2188 & 7.8125 & 2.24219 \\
   9$_{12}$ & 50.7813 & 18.75 & 10.1563 & 2.16797 \\
   9$_{13}$ & 48.4375 & 21.0938 & 7.8125 & 2.26172 \\
   9$_{14}$ & 50.7813 & 14.4531 & 14.4531 & 2.23047 \\
   9$_{15}$ & 52.3438 & 18.75 & 11.7188 & 2.05078 \\
   9$_{16}$ & 43.3594 & 29.2969 & 3.51563 & 2.47266 \\
   9$_{17}$ & 46.4844 & 18.75 & 12.5 & 2.37109 \\
   9$_{18}$ & 49.6094 & 20.7031 & 8.20313 & 2.19531 \\
   9$_{19}$ & 52.3438 & 15.2344 & 15.2344 & 2.11328 \\
   9$_{20}$ & 45.7031 & 24.2188 & 7.42188 & 2.35547 \\
   9$_{21}$ & 50.7813 & 17.5781 & 11.3281 & 2.1875 \\
   9$_{22}$ & 44.1406 & 21.0938 & 11.7188 & 2.53125 \\
   9$_{23}$ & 49.2188 & 22.2656 & 7.8125 & 2.1875 \\
   9$_{24}$ & 44.1406 & 25.7813 & 7.03125 & 2.44531 \\
   9$_{25}$ & 48.4375 & 20.3125 & 10.9375 & 2.32031 \\
   9$_{26}$ & 45.7031 & 19.5313 & 12.1094 & 2.42578 \\
   9$_{27}$ & 44.9219 & 22.2656 & 8.59375 & 2.43359 \\
   9$_{28}$ & 42.1875 & 28.9063 & 3.51563 & 2.51953 \\
   9$_{29}$ & 42.1875 & 19.5313 & 10.1563 & 2.76953 \\
   9$_{30}$ & 43.3594 & 24.2188 & 8.20313 & 2.54688 \\
   9$_{31}$ & 42.5781 & 25.7813 & 4.6875 & 2.53516 \\
   9$_{32}$ & 42.5781 & 19.5313 & 10.9375 & 2.73438 \\
   9$_{33}$ & 41.0156 & 22.6563 & 7.8125 & 2.77344 \\
   9$_{34}$ & 39.8438 & 15.625 & 13.2813 & 3.03516 \\
   9$_{35}$ & 52.7344 & 14.0625 & 10.5469 & 2.17969 
\end{tabular}
\end{table}

$*=$ foil

$\dagger = 2 \, n$ 

$\ddagger = k \, n$ 

$\mathsection$ = 2 1 $n$

\newpage

\begin{table}[h]
\begin{tabular}{c|c|c|c|c}
Knot & Unknot \% & Trefoil \% & Figure Eight \% & Expectation Value \\
 \hspace{5 mm} 9$_{36}$  \hspace{5 mm} & 45.3125 & 25.7813 & 7.03125 & 2.39453 \\
   9$_{37}$ & 49.2188 & 14.8438 & 14.8438 & 2.33203 \\
   9$_{38}$ & 44.5313 & 21.0938 & 7.03125 & 2.63281 \\
   9$_{39}$ & 45.3125 & 17.1875 & 10.9375 & 2.61719 \\
   9$_{40}$ & 36.3281 & 12.8906 & 15.2344 & 3.32422 \\
   9$_{41}$ & 44.1406 & 14.0625 & 14.0625 & 2.71875 \\
   9$_{42}$ & 45.3125 & 25.7813 & 7.03125 & 2.39453 \\
   9$_{43}$ & 43.3594 & 24.2188 & 8.20313 & 2.54688 \\
   9$_{44}$ & 41.0156 & 22.6563 & 7.8125 & 2.77344 \\
   9$_{45}$ & 44.1406 & 21.0938 & 11.7188 & 2.53125 \\
   9$_{46}$ & 52.7344 & 14.0625 & 10.5469 & 2.17969 \\
   9$_{47}$ & 39.8438 & 15.625 & 13.2813 & 3.03516 \\
   9$_{48}$ & 45.3125 & 17.1875 & 10.9375 & 2.61719 \\
   9$_{49}$ & 45.3125 & 17.1875 & 10.9375 & 2.61719
\end{tabular}
\end{table}

\begin{table}[h]
\begin{tabular}{c|c|c|c|c}
Knot & Unknot \% & Trefoil \% & Figure Eight \% & Expectation Value \\
$3_1\#3_1$ & 56.25 & 37.5 & 0. & 1.5 \\
   $3_1\#4_1$ & 56.25 & 28.125 & 9.375 & 1.625 \\ 
   $4_1\#4_1$ & 56.25 & 18.75 & 18.75 & 1.75 \\
   $3_1\#5_1$ & 46.875 & 39.0625 & 0. & 2. \\
   $3_1\#5_2$ & 51.5625 & 31.25 & 4.6875 & 1.875 \\
   $4_1\#5_1$ & 46.875 & 31.25 & 7.8125 & 2.125 \\
   $4_1\#5_2$ & 51.5625 & 22.6563 & 13.2813 & 2. \\
   $3_1\#6_1$ & 51.5625 & 26.5625 & 9.375 & 1.96875 \\
   $3_1\#6_2$ & 44.5313 & 31.25 & 7.03125 & 2.28125 \\
   $3_1\#6_3$ & 42.1875 & 35.1563 & 2.34375 & 2.375 \\
   $3_1\#3_1\#3_1$ & 42.1875 & 42.1875 & 0. & 2.25 
\end{tabular}
\end{table}

\end{document}

%% file: Intro.tex

Knots and links have been objects of mathematical interest for centuries.  In 1833, Gauss found the linking integral for two loops.  Knots and links were some of the first objects to be studied topologically.  They remain a cornerstone of the field of topology, and are useful in many settings beyond their inate one/three dimensional-ness.  

The combinatorial study of knots dates back to Reidemeister, who described a set of three moves which generate all equivalences of knot diagrams.  The major advantage of these moves is not in their direct application to questions of knot diagram equality; it is in knot invariants.  Showing that a quantity which can be computed from a knot diagram is invariant under Reidemeister moves is an easy (sometimes) way to show it is actually a property of knots.  

The Jones polynomial is one of the most famous invariants of knots.  It is calculated by `resolving' each crossing in a knot so that the strings no longer cross.  As there are two ways to do this for each crossing, $2^n$ possible crossing-free `states' (diagrams consisting only of disjoint loops) result.  The Jones polynomial is a weighted sum of such states (actually, of a polynomial quantity assigned to each state based on the number of loops it contains).  

In this article, we investigate a combinatorial aspect of knot theory coming from considering knot diagrams without crossing data, and the knots that result from random assignment of crossing data. 
Suppose we have a ``free knot diagram," like so:
\begin{center}
\begin{tikzpicture}
	\coordinate (A) at (0:1cm) {};
	\coordinate (B) at (120:1cm) {};
	\coordinate (C) at (240:1cm) {};
	\coordinate (D) at (0:2cm) {};
	\coordinate (E) at (120:2cm) {};
	\coordinate (F) at (240:2cm) {};
	
	\draw (A) .. controls +(90:.8cm) and +(30:.8cm) .. (E);
	\draw (C) .. controls +(-30:.8cm) and +(270:.8cm) .. (D)  .. controls +(90:.8cm) and +(30:.8cm) .. (B);
	\draw (B) .. controls +(210:.8cm) and +(150:.8cm) .. (F)  .. controls +(-30:.8cm) and +(-90:.8cm) .. (A);
	\draw (E)  .. controls +(210:.8cm) and +(150:.8cm) .. (C);
\end{tikzpicture}
\end{center}
If we randomly assign over/under data to each crossing in a knot diagram shape, how many unknots will we get?  For a generic free knot diagram, will any assignment of crossings produce trefoils, figure eight knots, the knot $5_2$, etc?  What is the average crossing number of all knots produced by such assignments to a given free knot diagram?

We began this project by taking an experimental approach:  For a set of free knot diagrams with a smallish number of crossings, we computed all the knots which result from assignments of crossing data.  We did this with the help of a Mathematica program and the Jones polynomial:  though not a perfect invariant, the Jones polynomial can tell apart all prime knots with $9$ or fewer crossings. We experience some computational savings from the fact that  knots coming from the same free knot diagram all resolve into the same states -- it is only the weights of the various states that change, depending on the crossing data.  

After looking at the data, we made some observations, and in this article prove many of them.  As a starting point, we have the following:  

\begin{theorem}
A random knot coming from an $n$-crossing knot diagram has probability at least $\frac{2n}{2^n}$ of being the unknot.
\end{theorem}

This theorem sets a lower bound on the possible amount of unknots, which is intriguing, but does not tell us much information about the vast majority of knots. For example, once we start looking at relatively uncomplicated diagrams with six crossings, this theorem only describes one third of all produced unknots. We want to establish stronger bounds.

Another observation we quickly made was the following:

\begin{theorem}
Let $K$ be a free knot diagram which is non-trivial.  Then $K$ has some assignment of crossings that produces a trefoil diagram.  
\end{theorem}

We examine four particular categories of free knot diagrams based on the closure of particular tangles, and completely describe the knots that result from two while also partially deducing the results of the two further free diagrams. Computational evidence suggests that some of these knots realize upper or lower bounds on unknots, trefoils, or figure eight knots.

\begin{conjecture} \label{trefoil conjecture}
The free foil diagram with n crossings, the free closure of the n tangle pictured below, realize the upper bound on trefoils for all sufficiently complicated free diagrams with n and $n+1$ crossings.
\end{conjecture}

\begin{center}
\begin{tikzpicture}
\begin{scope}[xshift=-4cm, yshift=1cm]
\draw (.25,0) -- (0,0);
\draw (.25,.5) -- (0,.5);
\draw (.25, 0) -- (.75, .5);
\draw (.25,.5) -- (.75,0);
\draw (.75,0)--(1,0);
\draw (.75, .5) -- (1,.5);

\node[draw=none] (ellipsis1) at (1.425,0.25) {$\cdots$};

\draw (1.75,0) -- (2,0);
\draw (1.75,.5) -- (2,.5);
\draw (2, 0) -- (2.5, .5);
\draw (2,.5) -- (2.5,0);
\draw (2.5,0)--(2.75,0);
\draw (2.5, .5) -- (2.75,.5);

\draw (2.75,0) arc [start angle = 270, end angle = 90, radius = -1mm];
\draw (0,0) arc [start angle = 90, end angle = 270, radius = 1mm];
\draw (2.75,.5) arc [start angle = 90, end angle = 270, radius = -1mm];
\draw (0,.5) arc [start angle = 270, end angle = 90, radius = 1mm];
\draw (0,0.7) -- (2.75, 0.7);
\draw (0,-0.2) -- (2.75, -0.2);
\draw [decorate,decoration={brace,amplitude=10pt},xshift=0pt,yshift=-4pt]  (2.5,-0.2)--(0.25,-0.2) node [black,midway,yshift=-0.6cm]{$n$};
\end{scope}
\end{tikzpicture}
\end{center}

\begin{conjecture}
The free foil diagram with n crossings realize the upper bound on figure eights for all sufficiently complicated algebraic free diagrams with n and $n+1$ crossings.
\end{conjecture}

\begin{conjecture} \label{unknotconjecture}
The following 2 m diagram, the free closure of the 2 m tangle, realizes the lower bound on the trefoils and the upper bound on the unknots for all sufficiently complicated free diagrams with m crossings.
\end{conjecture}

\begin{center}
\begin{tikzpicture}
\draw (-.25, 2) -- (.25, 2.5);
\draw (-.25, 2.5) -- (-.1, 2.35);
\draw (.25, 2) -- (.1, 2.15);

\draw (.25, 1.75) .. controls (.3, 1.875) .. (.25, 2);
\draw (-.25, 1.75) .. controls (-.3, 1.875) .. (-.25, 2);

\draw (-.25, 1.75) -- (-.1, 1.6);
\draw (-.25, 1.25) -- (.25, 1.75);
\draw (.25, 1.25) -- (.1, 1.4);

\draw (-.25, 1.25) .. controls (-.35, 1.15) and (-2, .6) .. (-1.42,0);
\draw (.25, 1.25) .. controls (.35, 1.15) and (2, .6) .. (1.42,0);

\draw (-.25, 2.5) .. controls (-1, 3) and (-2.25, .15) .. (-1.42, -.5);
\draw (.25, 2.5) .. controls (1, 3) and (2.25, .15) .. (1.42, -.5);

\begin{scope}[yscale = -1]
\draw (-1.17,0) -- (-1.42,0);
\draw (-1.17,.5) -- (-1.42,.5);
\draw (-1.17, 0) -- (-1.02, .15);
\draw (-1.17,.5) -- (-.67,0);
\draw (-.67, .5) -- (-.82, .35);
\draw (-.67,0)--(-.42,0);
\draw (-.67, .5) -- (-.42,.5);

\node[draw=none] (ellipsis1) at (0.05,0.25) {$\cdots$};

\draw (1.17,0) -- (1.42,0);
\draw (1.17,.5) -- (1.42,.5);
\draw (1.17, 0) -- (.67, .5);
\draw (.82,.15) -- (.67,0);
\draw (1.02, .35) -- (1.17, .5);
\draw (.67,0) -- (.42,0);
\draw (.67, .5) -- (.42,.5);
\end{scope}

\draw [decorate,decoration={brace,amplitude=10pt},xshift=0pt,yshift=-4pt]  (1.25,-0.5)--(-1.25,-0.5) node [black,midway,yshift=-0.6cm]{$m$};

\end{tikzpicture}
\end{center}

Based on the determined formula for unknots produced by a randomization of a free 2 $n$ knot, we also propose an absolute maximum of resultant unknot probability coming from a nontrivial free knot diagram of 0.75.

The structure of this article is as follows.  
Section \ref{background} gives the necessary knot theory background and defines free knot diagrams.  
Our general results are in Section \ref{results}.  
The most interesting of these is that every free knot diagram produces trefoils.  
Section \ref{tangles} reminds the reader of Conway's tangle notation and computes the complete resolution of $n$-foil knots (a family consisting of knots made by twisting two strands an odd number of times, including the trefoil as its simplest member).  
Section \ref{nk} computes the unknot probability, and many other knot probabilities, for the $k$  $n$ tangle knots.
Section \ref{21n} does the same for the $2$ $1$ $n$ tangle knots.  

At the end of each section, we include some conjectures which are supported by the data we generated, but do not seem accessible to prove at the moment.  Future directions to investigate include these conjectures and working out resolutions of other knot families.  It would be great to develop a more theoretical understanding of the role that various structure plays in knot resolutions:  understanding tangle structure and its role in generating (apparently) minimum-unknot and maximal-unknot examples, and understanding braid structure and why the figure eight knot appears to be universal among prime knot diagrams with braid index 3 and higher.

%% file: Background.tex

A knot is a smooth embedding of the circle $\mathbb{S}^1$ into Euclidean three-space $\mathbb{R}^3$, considered up to isotopy.  The mathematical study of knots, however, quickly turns into a study of essentially two-dimensional objects like so:  
\begin{center}
\begin{tikzpicture}
	\coordinate (A) at (0:1cm) {};
	\coordinate (B) at (120:1cm) {};
	\coordinate (C) at (240:1cm) {};
	\coordinate (D) at (0:2cm) {};
	\coordinate (E) at (120:2cm) {};
	\coordinate (F) at (240:2cm) {};
	
	\draw (A) .. controls +(90:.8cm) and +(30:.8cm) .. (E);
	\draw[line width=5pt, white] (D)  .. controls +(90:.8cm) and +(30:.8cm) .. (B);
	\draw (C) .. controls +(-30:.8cm) and +(270:.8cm) .. (D)  .. controls +(90:.8cm) and +(30:.8cm) .. (B);
	\draw [line width=5pt, white] (F)  .. controls +(-30:.8cm) and +(-90:.8cm) .. (A);
	\draw (B) .. controls +(210:.8cm) and +(150:.8cm) .. (F)  .. controls +(-30:.8cm) and +(-90:.8cm) .. (A);
	\draw[line width=5pt, white] (E)  .. controls +(210:.8cm) and +(150:.8cm) .. (C);
	\draw (E)  .. controls +(210:.8cm) and +(150:.8cm) .. (C);
\end{tikzpicture}
\end{center}
This is a \emph{knot diagram}: a smooth projection of a knot into $\mathbb{R}^2$ in which all crossings are transverse (not tangent) and involve only two strands, and the (barely) 3D data of which strand goes over is denoted by a break in the understand.  A \emph{strand} of a knot is the image under the embeddings of any interval of the original circle.  A \emph{link} is the multi-component generalization of a knot.

Given two knot diagrams, how do we know if they represent the same knot?  We may attempt to directly manipulate one diagram to turn it into the other one. Reidemeister moves are three isotopies between diagrams which generate all  isotopies:

\begin{center}
\begin{tikzpicture}
\draw (-1.75, 1.5) node {RI};

\draw[thick] (-3,0.75) -- (-3, -0.75);

\draw [thick] (-2,-.5) .. controls (-1.4, .4) .. (-1,0.45);
\draw [thick] (-2,.5) -- (-1.7, 0.1);
\draw [thick] (-1.575, -0.075) .. controls (-1.3, -0.45) .. (-1, -0.45);

\draw [thick, <->] (-2.75,0) -- (-2.25,0);

\draw [thick] (-1, -0.45) .. controls (-0.6, -.45) .. (-.575, 0);
\draw [thick] (-1, 0.45) .. controls (-0.6, .45) .. (-.575, 0);

\draw (2, 1.5) node {RII};

\draw[thick] (0.5,1) .. controls (1.5,0) .. (0.5,-1);
\draw [thick] (1.5,1) -- (1.05, 0.55);
\draw [thick] (1.5,-1) -- (1.05, -0.55);
\draw [thick] (0.92, 0.43) .. controls (0.5,0) .. (0.92, -0.43);

\draw [thick, <->] (1.75,0) -- (2.25,0);

\draw [thick] (2.5,1) .. controls (3,0) .. (2.5,-1);
\draw [thick] (3.5,1) .. controls (3,0) .. (3.5,-1);

\end{tikzpicture}

\begin{tikzpicture}
\draw (0, 1.5) node {RIII};

\draw [thick] (-2.5,1) -- (-1.65, 0.15);
\draw [thick] (-2.5,-1) -- (-.5,1);
\draw [thick] (-.5,-1) -- (-1.35, -0.15); 

\draw [thick, <->]  (-0.25, 0) -- (.25,0);

\draw [thick] (0.5,1) -- (1.35, 0.15);
\draw [thick] (0.5,-1) -- (2.5,1);
\draw [thick] (2.5,-1) -- (1.625, -0.15); 

\draw [thick] (-1.575, 1) -- (-1.98, 0.63);
\draw [thick] (-2.15, 0.5) .. controls (-2.75, 0) .. (-2.15, -0.5);
\draw [thick] (-1.575, -1) -- (-1.98, -0.63);

\draw [thick] (1.575, 1) -- (1.98, 0.63);
\draw [thick] (2.15, 0.5) .. controls (2.75, 0) .. (2.15, -0.5);
\draw [thick] (1.575, -1) -- (1.98, -0.63);

\end{tikzpicture}
\end{center}
The benefit of the Reidemeister moves is not in their direct application, which is tedious, but in proving that  invariants defined on knot diagrams lift to knots themselves.

The Jones polynomial, the second  polynomial knot invariant to be discovered, was originally observed by Jones in the context of braid group representations \cite{MR766964}.  We present it by a two-step diagrammatic definition due to Kauffman.  First we will define the Kauffman bracket, following  \cite{MR899057}:

\begin{definition}
The Kauffman bracket is a (Laurent) polynomial in the variable $A$.  It is the result of repeatedly applying a `crossing-resolving' relation which smooths crossings in both possible ways, combined with the relation that any closed and unlinked loop is counted by multiplying by $(-A^2-A^{-2})$.

$$
\left<
\begin{tikzpicture}[baseline=.4cm]
	\draw[thick] (0,0)--(1,1);
	\draw[thick] (0,1)--(.4,.6);
	\draw[thick] (1,0)--(.6,.4);
\end{tikzpicture}
\right>
=
A
\left<
\begin{tikzpicture}[baseline=.4cm]
	\draw[thick] (0,0) .. controls (.5,.5) .. (0,1);
	\draw[thick] (1,0) .. controls (.5,.5) .. (1,1);
\end{tikzpicture}
\right>
+
A^{-1}
\left<
\begin{tikzpicture}[baseline=.4cm]
	\draw[thick] (0,0) .. controls (.5,.5) .. (1,0);
	\draw[thick] (0,1) .. controls (.5,.5) .. (1,1);
\end{tikzpicture}
\right>
$$

\newcommand*{\Scale}[2][4]{\scalebox{#1}{\ensuremath{#2}}}

$$
\left<
\begin{tikzpicture}[baseline=0.4cm]
	\draw[thick] (.5,.5) circle (.25);
\end{tikzpicture}
\cup
L
\right>
=
(-A^2-A^{-2})
\left<
L
\right>
\hspace{1cm}
\left<
\begin{tikzpicture}[baseline=0.4cm]
	\draw[thick] (.5,.5) circle (.25);
\end{tikzpicture}
\right>
=
1
$$
\end{definition}

\begin{example}
\input{KauffmanEx.tex}
\end{example}

By definition, the Kauffman bracket of an untwisted loop is one, but the knot diagram above has a different Kauffman bracket despite being another representation of the unknot. This is because the Kauffman bracket is invariant under Reidemeister II and III, but not under Reidemeister I.  To fix that, we make use of the writhe of a knot:

\begin{definition}
Choose an orientation for a knot $K$. Then match one of the below crossing types and assign the corresponding value to each crossing. The \emph{writhe $w(K)$} is the sum of the values at all crossings.
\end{definition}

\begin{center}
\begin{tikzpicture}
\begin{scope}
\draw (0, 1.25) node {+1};
\draw [->] (-.15, .15) -- (-1,1);
\draw [->] (-1,-1) -- (1,1);
\draw (.15, -.15) -- (1,-1);
\end{scope}
\begin{scope}[xshift=4cm]
\draw (0, 1.25) node {-1};
\draw [->] (1, -1) -- (-1,1);
\draw [->] (.15,.15) -- (1,1);
\draw (-.15, -.15) -- (-1,-1);
\end{scope}
 
\end{tikzpicture}
\end{center}

For a knot, the choice of orientation is arbitrary: reversing it will reverse the direction of \emph{both} arrows coming out of a crossing, thus preserving the sign of the writhe at that crossing. The writhe fails to be invariant under Reidemeister I, yet we can make the Kauffman bracket and writhe's failures cancel each other out,  creating a true invariant; this is the Jones polynomial.

\begin{definition}
For a knot K, the \emph{Jones polynomial} $V_K= \left< K \right> (A)^{-3 w(K)}$ is a knot invariant.
\end{definition}

Another knot invariant is the unknotting number. It is the least number of crossings one can change in any knot diagram to unknot it. This change looks like so, where the ``overstrand" becomes the ``understrand" and vice versa.

\begin{center}
\begin{tikzpicture}
\draw (-3, 0) node {Overstrand};
\draw [->, thick] (-1.9,-.05) -- (-1,-0.5);
\draw (6.5, 0) node {Understrand};
\draw [->, thick] (5.3, 0) -- (4.5,0.5);
\draw  (-.15, .15) -- (-1,1);
\draw (-1,-1) -- (1,1);
\draw (.15, -.15) -- (1,-1);
\draw [<->, thick] (1.5, 0) -- (2,0);
\draw  (4.5, -1) -- (2.5,1);
\draw  (3.65,.15) -- (4.5,1);
\draw (3.35, -.15) -- (2.5,-1);
\end{tikzpicture}
\end{center}

While easy to compute for a particular diagram, the difficulty arises when attempting to show that no other projection of the knot can be unknotted with fewer changes. Note that altering the crossing can not turn a knot into a link as swapping the crossing requires cutting one of the strands and, after rethreading, connecting it back to the same location as before.

The question we consider in this article, of which knots come from which free knot diagrams, is in some sense a converse to the question of unknotting number.  Instead of starting with a knot diagram and trying to change it to reach the unknot, we start with only the shape of a knot diagram, and ask where we may land by assigning its crossings.

Random knots have been studied before, usually from a geometric point of view relating to their appearance in random walks and polymers \cite{MR1981020}.
For diagrammatic random knots, some information is known about their average writhe \cite{MR2740363}, but the question of what percent of their assignments are unknots has not been answered.

\begin{example}
Here is a free knot diagram which can produce the trefoil knot and the unknot:
\begin{center}
\begin{tikzpicture}
	\coordinate (A) at (0:1cm) {};
	\coordinate (B) at (120:1cm) {};
	\coordinate (C) at (240:1cm) {};
	\coordinate (D) at (0:2cm) {};
	\coordinate (E) at (120:2cm) {};
	\coordinate (F) at (240:2cm) {};
	
	\draw (A) .. controls +(90:.8cm) and +(30:.8cm) .. (E);
	\draw (C) .. controls +(-30:.8cm) and +(270:.8cm) .. (D)  .. controls +(90:.8cm) and +(30:.8cm) .. (B);
	\draw (B) .. controls +(210:.8cm) and +(150:.8cm) .. (F)  .. controls +(-30:.8cm) and +(-90:.8cm) .. (A);
	\draw (E)  .. controls +(210:.8cm) and +(150:.8cm) .. (C);
\end{tikzpicture}
\end{center}
\end{example}

By assigning crossings in such a diagram, we may then use Reidemeister moves to simplify the knots. Here is an example assignment and simplification of the free trefoil.

\begin{tikzpicture}
	\coordinate (A) at (0:1cm) {};
	\coordinate (B) at (120:1cm) {};
	\coordinate (C) at (240:1cm) {};
	\coordinate (D) at (0:2cm) {};
	\coordinate (E) at (120:2cm) {};
	\coordinate (F) at (240:2cm) {};
	
	\draw (A) .. controls +(90:.8cm) and +(30:.8cm) .. (E);
	\draw (C) .. controls +(-30:.8cm) and +(270:.8cm) .. (D)  .. controls +(90:.8cm) and +(30:.8cm) .. (B);
	\draw (B) .. controls +(210:.8cm) and +(150:.8cm) .. (F)  .. controls +(-30:.8cm) and +(-90:.8cm) .. (A);
	\draw (E)  .. controls +(210:.8cm) and +(150:.8cm) .. (C);

\draw [->] (2.25,0) -- (2.75, 0);

\begin{scope}[xshift=4.25cm]
	\coordinate (A) at (0:1cm) {};
	\coordinate (B) at (120:1cm) {};
	\coordinate (C) at (240:1cm) {};
	\coordinate (D) at (0:2cm) {};
	\coordinate (E) at (120:2cm) {};
	\coordinate (F) at (240:2cm) {};
	
	\draw (A) .. controls +(90:.8cm) and +(30:.8cm) .. (E);
	\draw[line width=5pt, white] (D)  .. controls +(90:.8cm) and +(30:.8cm) .. (B);
	\draw (C) .. controls +(-30:.8cm) and +(270:.8cm) .. (D);
	\draw (D)  .. controls +(90:.8cm) and +(30:.8cm) .. (B);
	\draw (F)  .. controls +(-30:.8cm) and +(-90:.8cm) .. (A);
	\draw [line width=5pt, white] (C)  .. controls +(-30:.8cm) and +(-90:.8cm) .. (D);
	\draw (C) .. controls +(-30:.8cm) and +(270:.8cm) .. (D);
	\draw (B) .. controls +(210:.8cm) and +(150:.8cm) .. (F);
	\draw[line width=5pt, white] (E)  .. controls +(210:.8cm) and +(150:.8cm) .. (C);
	\draw (E)  .. controls +(210:.8cm) and +(150:.8cm) .. (C);
\end{scope}

\draw [<->] (6.5, 0) -- (7,0);

\begin{scope}[xshift=8.5cm]
	\coordinate (A) at (0:1cm) {};
	\coordinate (B) at (120:1cm) {};
	\coordinate (C) at (240:1cm) {};
	\coordinate (D) at (0:2cm) {};
	\coordinate (E) at (120:2cm) {};
	\coordinate (F) at (240:2cm) {};
	\coordinate (G) at (0:0.3cm) {};
	
	\draw (A) .. controls +(90:.8cm) and +(30:.8cm) .. (E);
	\draw [line width=5pt, white] (C)  .. controls +(-30:.8cm) and +(-90:.8cm) .. (D);
	\draw (B) .. controls +(210:.8cm) and +(150:.8cm) .. (F)  .. controls +(-30:.8cm) and +(-90:.8cm) .. (A);
	\draw[line width=5pt, white] (E)  .. controls +(210:.8cm) and +(150:.8cm) .. (C);
	\draw (E)  .. controls +(210:.8cm) and +(150:.8cm) .. (C);
	\draw (B) .. controls (0.25, 1.1) .. (G);
	\draw (C) .. controls (0.25, -1.1) .. (G);
\end{scope}
\end{tikzpicture}

Similar simplification with different assignments shows the free trefoil diagram's 8 resultants are 6 unknots and 2 trefoils. 

In order to investigate the combinatorial properties of `crossingless' knots, we define a free knot diagram, meaning a projection into the plane of a knot which does not record which strands go over which others at crossings.  It is more convenient, however, to say the following:

\begin{definition}
A \emph{free knot diagram} is a planar 4-valent graph, considered up to planar isotopy (i.e. continuous deformations in the plane which preserve vertices and edges).   The 4-valent vertices are called \emph{free crossings}.
\end{definition}

We may produce a free knot diagram from any knot diagram, by forgetting the over/under information of the crossings:

\begin{definition}
The \emph{shape} of a knot diagram is the free knot diagram which results from making all the crossings of the original knot diagram  into free crossings.
\end{definition}

Reidemeister moves do not apply to free knot diagrams, since planar isotopies of free knot diagrams  preserve crossings. A Reidemeister move (incorrectly) applied to a free knot diagram  $S_1$ would create a distinct free knot diagram $S_2$. We can move back to more familiar ground by ``assigning" crossings:

\begin{definition}
When we \emph{assign} a free crossing, we choose an overstrand and an understrand at that crossing.   An \emph{assignment} of a free knot is a choice of how to assign each crossing.
\end{definition}

\begin{definition}
A \emph{mixed knot diagram} is the projection into the plane of a knot, which records over/under information for only some crossings.
\end{definition}

\begin{example}
A free knot diagram (left) and a mixed knot diagram (right).\\
\begin{center}
\begin{tikzpicture}
	\coordinate (A) at (0:1cm) {};
	\coordinate (B) at (120:1cm) {};
	\coordinate (C) at (240:1cm) {};
	\coordinate (D) at (0:2cm) {};
	\coordinate (E) at (120:2cm) {};
	\coordinate (F) at (240:2cm) {};
	
	\draw (A) .. controls +(90:.8cm) and +(30:.8cm) .. (E);
	\draw (C) .. controls +(-30:.8cm) and +(270:.8cm) .. (D)  .. controls +(90:.8cm) and +(30:.8cm) .. (B);
	\draw (B) .. controls +(210:.8cm) and +(150:.8cm) .. (F)  .. controls +(-30:.8cm) and +(-90:.8cm) .. (A);
	\draw (E)  .. controls +(210:.8cm) and +(150:.8cm) .. (C);
\end{tikzpicture}
\begin{tikzpicture}
	\coordinate (A) at (0:1cm) {};
	\coordinate (B) at (120:1cm) {};
	\coordinate (C) at (240:1cm) {};
	\coordinate (D) at (0:2cm) {};
	\coordinate (E) at (120:2cm) {};
	\coordinate (F) at (240:2cm) {};
	
	\draw (A) .. controls +(90:.8cm) and +(30:.8cm) .. (E);
	\draw[line width=5pt, white] (D)  .. controls +(90:.8cm) and +(30:.8cm) .. (B);
	\draw (C) .. controls +(-30:.8cm) and +(270:.8cm) .. (D)  .. controls +(90:.8cm) and +(30:.8cm) .. (B);
	\draw [line width=5pt, white] (F)  .. controls +(-30:.8cm) and +(-90:.8cm) .. (A);
	\draw (B) .. controls +(210:.8cm) and +(150:.8cm) .. (F)  .. controls +(-30:.8cm) and +(-90:.8cm) .. (A);
	\draw (E)  .. controls +(210:.8cm) and +(150:.8cm) .. (C);
\end{tikzpicture}
\end{center}
\end{example}

In a mixed knot diagram, we may freely apply Reidemeister moves whenever they involve genuine crossings and no free crossings.

In exploring the question of ``which knots does this free knot diagram produce?," we assign every possible combination of crossings on a free knot diagram. The two choices at each crossing ensure that for a free knot diagram with $n$ crossings, there are $2^n$ total resultants. The computation of all resultants creates a sample space of outcomes for a random assignment of crossings for a free knot diagram.

\begin{definition}
The \emph{resultant knot probability} for a given knot $R$ from a free knot diagram $F$ is the probability that the randomization process applied to $F$ will produce $R$.
\end{definition}

Whenever a resultant knot is produced, its mirror image is also produced.  Thus, we do not distinguish knots and their mirror images, and the minimum resultant knot probability for an $n$-crossing free knot diagram is $\frac{2}{2^n}$.

The direct calculation of resultant knot probabilities is intensive, and computation time increases quickly as $n$ increases.
In the appendix, we list the probabilities of the unknot, trefoil, and figure eight knots, in addition to the expectation value, for the knot shapes of the knots with at most nine crossings. The question of what bounds can we place on resultants in general is explored in the next section.

Readers may wonder if there is a connection between our free knot diagrams, and virtual knot theory \cite{MR1721925}. 
Notice, however, that in virtual knot theory the second Reidemeister move holds, while in our context, there is no second Reidemeister move involving free crossings.

\begin{center}
\begin{tikzpicture}
\draw[thick] (-3,1) .. controls (-2,0) .. (-3,-1);
\draw[thick] (-2, 1) .. controls (-3, 0) .. (-2, -1);
\draw[thick] (-2.5, 0.475) circle (.2);
\draw[thick] (-2.5, -0.475) circle (.2);

\draw [thick, <->] (-1.75,0) -- (-1.25,0);

\draw [thick] (-1,1) .. controls (-0.5,0) .. (-1,-1);
\draw [thick] (0,1) .. controls (-.5,0) .. (0,-1);
\draw[thick] (1,1) .. controls (2,0) .. (1,-1);
\draw[thick] (2, 1) .. controls (1, 0) .. (2, -1);

\draw [thick, <->] (2.25,0) -- (2.75,0);
\draw [thick, red] (2.25, .25) -- (2.75, -.25);
\draw [thick, red] (2.75, .25) -- (2.25, -.25);

\draw [thick] (3,1) .. controls (3.5,0) .. (3,-1);
\draw [thick] (4,1) .. controls (3.5,0) .. (4,-1);

\end{tikzpicture}
\end{center}

%% file: KauffmanEx.tex
Let's work out the Kauffman bracket of an unknot with the same `shape' as the trefoil knot.
$$
\left<
\begin{tikzpicture}[baseline=0cm]
\begin{scope}[scale=0.5]
	\coordinate (A) at (0:1cm) {};
	\coordinate (B) at (120:1cm) {};
	\coordinate (C) at (240:1cm) {};
	\coordinate (D) at (0:2cm) {};
	\coordinate (E) at (120:2cm) {};
	\coordinate (F) at (240:2cm) {};
	
	\draw (A) .. controls +(90:.8cm) and +(30:.8cm) .. (E);
	\draw[line width=5pt, white] (D)  .. controls +(90:.8cm) and +(30:.8cm) .. (B);
	\draw (C) .. controls +(-30:.8cm) and +(270:.8cm) .. (D);
	\draw (D)  .. controls +(90:.8cm) and +(30:.8cm) .. (B);
	\draw (F)  .. controls +(-30:.8cm) and +(-90:.8cm) .. (A);
	\draw [line width=5pt, white] (C)  .. controls +(-30:.8cm) and +(-90:.8cm) .. (D);
	\draw (C) .. controls +(-30:.8cm) and +(270:.8cm) .. (D);
	\draw (B) .. controls +(210:.8cm) and +(150:.8cm) .. (F);
	\draw[line width=5pt, white] (E)  .. controls +(210:.8cm) and +(150:.8cm) .. (C);
	\draw (E)  .. controls +(210:.8cm) and +(150:.8cm) .. (C);
\end{scope}
\end{tikzpicture}
\right>
=
A
\left<
\begin{tikzpicture}[baseline=0cm]
\begin{scope}[scale=0.5]
	\coordinate (A) at (0:1cm) {};
	\coordinate (B) at (120:1cm) {};
	\coordinate (C) at (240:1cm) {};
	\coordinate (D) at (0:2cm) {};
	\coordinate (E) at (120:2cm) {};
	\coordinate (F) at (240:2cm) {};
	
	\draw (A) .. controls +(90:.8cm) and +(30:.8cm) .. (E);
	\draw[line width=5pt, white] (D)  .. controls +(90:.8cm) and +(30:.8cm) .. (B);
	\draw (C) .. controls +(-30:.8cm) and +(270:.8cm) .. (D);
	\draw (D)  .. controls +(90:.8cm) and +(30:.8cm) .. (B);
	\draw (F)  .. controls +(-30:.8cm) and +(-90:.8cm) .. (A);
	\draw [line width=5pt, white] (C)  .. controls +(-30:.8cm) and +(-90:.8cm) .. (D);
	\draw (C) .. controls +(-30:.8cm) and +(270:.8cm) .. (D);
	\draw (B) .. controls +(210:.8cm) and +(150:.8cm) .. (F);
	\draw[line width=5pt, white] (E)  .. controls +(210:.8cm) and +(150:.8cm) .. (C);
	\draw (E)  .. controls +(210:.8cm) and +(150:.8cm) .. (C);

\draw[white, fill=white] (180: 1cm) circle (.75cm);
\draw (-.681, .725) .. controls (-1, .5)  and (-1.2, .25) ..  (-1.32,.7);
\draw (-.681, -.725) .. controls (-1, -.5)  and (-1.2, -.25) ..  (-1.32,-.7);
\end{scope}
\end{tikzpicture}
\right>
+A^{-1}
\left<
\begin{tikzpicture}[baseline=0cm]
\begin{scope}[scale=0.5]
	\coordinate (A) at (0:1cm) {};
	\coordinate (B) at (120:1cm) {};
	\coordinate (C) at (240:1cm) {};
	\coordinate (D) at (0:2cm) {};
	\coordinate (E) at (120:2cm) {};
	\coordinate (F) at (240:2cm) {};
	
	\draw (A) .. controls +(90:.8cm) and +(30:.8cm) .. (E);
	\draw[line width=5pt, white] (D)  .. controls +(90:.8cm) and +(30:.8cm) .. (B);
	\draw (C) .. controls +(-30:.8cm) and +(270:.8cm) .. (D);
	\draw (D)  .. controls +(90:.8cm) and +(30:.8cm) .. (B);
	\draw (F)  .. controls +(-30:.8cm) and +(-90:.8cm) .. (A);
	\draw [line width=5pt, white] (C)  .. controls +(-30:.8cm) and +(-90:.8cm) .. (D);
	\draw (C) .. controls +(-30:.8cm) and +(270:.8cm) .. (D);
	\draw (B) .. controls +(210:.8cm) and +(150:.8cm) .. (F);
	\draw[line width=5pt, white] (E)  .. controls +(210:.8cm) and +(150:.8cm) .. (C);
	\draw (E)  .. controls +(210:.8cm) and +(150:.8cm) .. (C);
	
\draw[white, fill=white] (180: 1cm) circle (.75cm);
\draw (-.695, .72) .. controls (-1., .5) and (-1, -.5) ..  (-.695, -.72);
\draw (-1.32,.7) .. controls (-1.2, 0)  ..  (-1.32,-.7);
\end{scope}
\end{tikzpicture}
\right>
$$

$$
A
\left<
\begin{tikzpicture}[baseline=0cm]
\begin{scope}[scale=0.5]
	\coordinate (A) at (0:1cm) {};
	\coordinate (B) at (120:1cm) {};
	\coordinate (C) at (240:1cm) {};
	\coordinate (D) at (0:2cm) {};
	\coordinate (E) at (120:2cm) {};
	\coordinate (F) at (240:2cm) {};
	
	\draw (A) .. controls +(90:.8cm) and +(30:.8cm) .. (E);
	\draw[line width=5pt, white] (D)  .. controls +(90:.8cm) and +(30:.8cm) .. (B);
	\draw (C) .. controls +(-30:.8cm) and +(270:.8cm) .. (D);
	\draw (D)  .. controls +(90:.8cm) and +(30:.8cm) .. (B);
	\draw (F)  .. controls +(-30:.8cm) and +(-90:.8cm) .. (A);
	\draw [line width=5pt, white] (C)  .. controls +(-30:.8cm) and +(-90:.8cm) .. (D);
	\draw (C) .. controls +(-30:.8cm) and +(270:.8cm) .. (D);
	\draw (B) .. controls +(210:.8cm) and +(150:.8cm) .. (F);
	\draw[line width=5pt, white] (E)  .. controls +(210:.8cm) and +(150:.8cm) .. (C);
	\draw (E)  .. controls +(210:.8cm) and +(150:.8cm) .. (C);

\draw[white, fill=white] (180: 1cm) circle (.75cm);
\draw (-.681, .725) .. controls (-1, .5)  and (-1.2, .25) ..  (-1.32,.7);
\draw (-.681, -.725) .. controls (-1, -.5)  and (-1.2, -.25) ..  (-1.32,-.7);
\end{scope}
\end{tikzpicture}
\right>
=
A
\left(
A
\left<
\begin{tikzpicture}[baseline=0cm]

\begin{scope}[scale=0.5]
	\coordinate (A) at (0:1cm) {};
	\coordinate (B) at (120:1cm) {};
	\coordinate (C) at (240:1cm) {};
	\coordinate (D) at (0:2cm) {};
	\coordinate (E) at (120:2cm) {};
	\coordinate (F) at (240:2cm) {};
	
	\draw (A) .. controls +(90:.8cm) and +(30:.8cm) .. (E);
	\draw[line width=5pt, white] (D)  .. controls +(90:.8cm) and +(30:.8cm) .. (B);
	\draw (C) .. controls +(-30:.8cm) and +(270:.8cm) .. (D);
	\draw (D)  .. controls +(90:.8cm) and +(30:.8cm) .. (B);
	\draw (F)  .. controls +(-30:.8cm) and +(-90:.8cm) .. (A);
	\draw [line width=5pt, white] (C)  .. controls +(-30:.8cm) and +(-90:.8cm) .. (D);
	\draw (C) .. controls +(-30:.8cm) and +(270:.8cm) .. (D);
	\draw (B) .. controls +(210:.8cm) and +(150:.8cm) .. (F);
	\draw[line width=5pt, white] (E)  .. controls +(210:.8cm) and +(150:.8cm) .. (C);
	\draw (E)  .. controls +(210:.8cm) and +(150:.8cm) .. (C);

\draw[white, fill=white] (180: 1cm) circle (.75cm);
\draw (-.681, .725) .. controls (-1, .5)  and (-1.2, .25) ..  (-1.32,.7);
\draw (-.681, -.725) .. controls (-1, -.5)  and (-1.2, -.25) ..  (-1.32,-.7);

\draw [white, fill=white] (50: 1cm) circle (.75cm);
\draw (-.09, .995) .. controls (.45, 1.1) .. (.18, 1.39);
\draw (1, .08) .. controls (.85, .9) .. (1.41, .725);
\end{scope}
\end{tikzpicture}
\right>
+
A^{-1}
\left<
\begin{tikzpicture}[baseline=0cm]
\begin{scope}[scale=0.5]
	\coordinate (A) at (0:1cm) {};
	\coordinate (B) at (120:1cm) {};
	\coordinate (C) at (240:1cm) {};
	\coordinate (D) at (0:2cm) {};
	\coordinate (E) at (120:2cm) {};
	\coordinate (F) at (240:2cm) {};
	
	\draw (A) .. controls +(90:.8cm) and +(30:.8cm) .. (E);
	\draw[line width=5pt, white] (D)  .. controls +(90:.8cm) and +(30:.8cm) .. (B);
	\draw (C) .. controls +(-30:.8cm) and +(270:.8cm) .. (D);
	\draw (D)  .. controls +(90:.8cm) and +(30:.8cm) .. (B);
	\draw (F)  .. controls +(-30:.8cm) and +(-90:.8cm) .. (A);
	\draw [line width=5pt, white] (C)  .. controls +(-30:.8cm) and +(-90:.8cm) .. (D);
	\draw (C) .. controls +(-30:.8cm) and +(270:.8cm) .. (D);
	\draw (B) .. controls +(210:.8cm) and +(150:.8cm) .. (F);
	\draw[line width=5pt, white] (E)  .. controls +(210:.8cm) and +(150:.8cm) .. (C);
	\draw (E)  .. controls +(210:.8cm) and +(150:.8cm) .. (C);

\draw[white, fill=white] (180: 1cm) circle (.75cm);
\draw (-.681, .725) .. controls (-1, .5)  and (-1.2, .25) ..  (-1.32,.7);
\draw (-.681, -.725) .. controls (-1, -.5)  and (-1.2, -.25) ..  (-1.32,-.7);

\draw [white, fill=white] (50: 1cm) circle (.75cm);
\draw (-.09, .995) .. controls (.55, 1.1) and (.9, .5) .. (1, .08);
\draw  (.18, 1.39) .. controls (.795, 1) .. (1.41, .725);

\end{scope}
\end{tikzpicture}
\right>
\right)
$$

$$
=
A^2
(-A^2-A^{-2})
\left<
\begin{tikzpicture}[baseline=-0.1cm]
\draw (-.55, 0) .. controls (-.5, .5) .. (-.25, .25);
\draw (-.55, 0) .. controls (-.5, -.5) .. (-.25, -.25); 
\draw (-.25, -.25) -- (-.1, -.1);
\draw (-.25, .25) -- (.25, -.25);
\draw (.25, .25) -- (.1, .1);
\draw (.25, .25) .. controls (.5, .5) .. (.55, 0);
\draw (.25, -.25) .. controls (.5, -.5) .. (.55,0);
\end{tikzpicture}
\right>
+
\left<
\begin{tikzpicture}[baseline=-0.1cm]
\draw (-.55, 0) .. controls (-.5, .5) .. (-.25, .25);
\draw (-.55, 0) .. controls (-.5, -.5) .. (-.25, -.25); 
\draw (-.25, -.25) -- (-.1, -.1);
\draw (-.25, .25) -- (.25, -.25);
\draw (.25, .25) -- (.1, .1);
\draw (.25, .25) .. controls (.5, .5) .. (.55, 0);
\draw (.25, -.25) .. controls (.5, -.5) .. (.55,0);
\end{tikzpicture}
\right>
=
-A^4
\left<
\begin{tikzpicture}[baseline=-0.1cm]
\draw (-.55, 0) .. controls (-.5, .5) .. (-.25, .25);
\draw (-.55, 0) .. controls (-.5, -.5) .. (-.25, -.25); 
\draw (-.25, -.25) -- (-.1, -.1);
\draw (-.25, .25) -- (.25, -.25);
\draw (.25, .25) -- (.1, .1);
\draw (.25, .25) .. controls (.5, .5) .. (.55, 0);
\draw (.25, -.25) .. controls (.5, -.5) .. (.55,0);
\end{tikzpicture}
\right>
$$

$$
=
-A^4
\left(
A
\left<
\begin{tikzpicture}[baseline=-0.1cm]
\draw (0,0) circle (0.5cm);
\end{tikzpicture}
\right>
+
A^{-1}
\left<
\begin{tikzpicture}[baseline=-0.1cm]
\draw (.75,0) circle (0.5 cm);
\draw (-.75,0) circle (0.5 cm);
\end{tikzpicture}
\right>
\right)
=
A
$$

Similarly,

$$
A^{-1}
\left<
\begin{tikzpicture}[baseline=0cm]
\begin{scope}[scale=0.5]
	\coordinate (A) at (0:1cm) {};
	\coordinate (B) at (120:1cm) {};
	\coordinate (C) at (240:1cm) {};
	\coordinate (D) at (0:2cm) {};
	\coordinate (E) at (120:2cm) {};
	\coordinate (F) at (240:2cm) {};
	
	\draw (A) .. controls +(90:.8cm) and +(30:.8cm) .. (E);
	\draw[line width=5pt, white] (D)  .. controls +(90:.8cm) and +(30:.8cm) .. (B);
	\draw (C) .. controls +(-30:.8cm) and +(270:.8cm) .. (D);
	\draw (D)  .. controls +(90:.8cm) and +(30:.8cm) .. (B);
	\draw (F)  .. controls +(-30:.8cm) and +(-90:.8cm) .. (A);
	\draw [line width=5pt, white] (C)  .. controls +(-30:.8cm) and +(-90:.8cm) .. (D);
	\draw (C) .. controls +(-30:.8cm) and +(270:.8cm) .. (D);
	\draw (B) .. controls +(210:.8cm) and +(150:.8cm) .. (F);
	\draw[line width=5pt, white] (E)  .. controls +(210:.8cm) and +(150:.8cm) .. (C);
	\draw (E)  .. controls +(210:.8cm) and +(150:.8cm) .. (C);
	
\draw[white, fill=white] (180: 1cm) circle (.75cm);
\draw (-.695, .72) .. controls (-1., .5) and (-1, -.5) ..  (-.695, -.72);
\draw (-1.32,.7) .. controls (-1.2, 0)  ..  (-1.32,-.7);
\end{scope}
\end{tikzpicture}
\right>
=
-A^{-3}-A
\implies
\left<
\begin{tikzpicture}[baseline=0cm]
\begin{scope}[scale=0.5]
	\coordinate (A) at (0:1cm) {};
	\coordinate (B) at (120:1cm) {};
	\coordinate (C) at (240:1cm) {};
	\coordinate (D) at (0:2cm) {};
	\coordinate (E) at (120:2cm) {};
	\coordinate (F) at (240:2cm) {};
	
	\draw (A) .. controls +(90:.8cm) and +(30:.8cm) .. (E);
	\draw[line width=5pt, white] (D)  .. controls +(90:.8cm) and +(30:.8cm) .. (B);
	\draw (C) .. controls +(-30:.8cm) and +(270:.8cm) .. (D);
	\draw (D)  .. controls +(90:.8cm) and +(30:.8cm) .. (B);
	\draw (F)  .. controls +(-30:.8cm) and +(-90:.8cm) .. (A);
	\draw [line width=5pt, white] (C)  .. controls +(-30:.8cm) and +(-90:.8cm) .. (D);
	\draw (C) .. controls +(-30:.8cm) and +(270:.8cm) .. (D);
	\draw (B) .. controls +(210:.8cm) and +(150:.8cm) .. (F);
	\draw[line width=5pt, white] (E)  .. controls +(210:.8cm) and +(150:.8cm) .. (C);
	\draw (E)  .. controls +(210:.8cm) and +(150:.8cm) .. (C);
\end{scope}
\end{tikzpicture}
\right>
=
-A^{-3}
$$

%% file: Results.tex
That any knot diagram can have some of its crossings changed to represent the unknot is well-known, as is the algorithm described below for doing so.  We include this proof as a warm up, and because it uses techniques we will draw on later.

\begin{theorem}
An $n$-crossing free knot diagram has a minimum of $2n$ assignments which are the unknot.
\end{theorem}

\begin{proof}
By an `arc' of a free knot diagram, we mean a portion of the string containing no crossings (that is, a subset which is homeomorphic to an interval).  

Choose a point on an arc, and make it the highest point of the knot (which we now think of as coming out of the page).  Now choose a direction of travel, and force our path to travel downhill from there.  At the first crossing we come to, since it is directly connected to the highest point on the diagram, make the strand we are travelling along the  overstrand. Similarly, at every undefined crossing we encounter, make the current strand go over the other strand until we have returned to our starting segment.  The end of the segment we have now arrived at is the lowest point of the knot.  We call the arc containing our initial point the `climb;' the rest of the knot is the `downramp.' 

In our previous viewpoint, we looked down on the knot from above.  Rotate the knot so that we now look at it from the side.  From here, we see that the knot we created this way is isotopic to the unknot:  Since every height, other than the very top and bottom points, has exactly two points at that height (one from the climb, the other from the downramp), the identification with a circle is straightforward.  

An $n$-crossing free knot diagram has $2n$ distinct arcs, and as each arc has two directions to choose, we have $4n$ distinct climb/downramp pairs on the free knot diagram.  But this does not mean we have $4n$ combinatorially distinct unknot diagrams!  An easy way to tell apart the unknot diagrams is by their `top track:'  This is the largest strand containing of the climb and only overcrossings.  In other words, it stops just before passing through any crossing for the second time.  An unknot diagram may share its top track with at most one other unknot diagram -- this diagram would traverse the same top track in the opposite order.  

Thus, at least $\frac{4n}{2}=2n$ of the $2^n$ different knot diagrams associated to an $n$-crossing free knot diagram are diagrams of the unknot.
\end{proof}

\begin{definition}
A \emph{loop} in a knot diagram or a free knot diagram is a segment of the knot which starts and ends at the same crossing, and does not cross itself otherwise.  The \emph{length} of a loop is the number of crossings that segment passes through (with the crossing of origin/terminus only counting once.)  All lengths of loops in knot diagrams are odd numbers.
\end{definition}

A loop of length one in a knot diagram can be `straightened out' via Reidemeister relation I.   A loop of length one in a free knot diagram is `straightenable' in the sense that, once over/under information is assigned to that crossing, a Reidemeister I move will straighten it out regardless of how the over/under information is assigned.

\begin{theorem}
Resultant knot probability for a free or mixed knot diagram is invariant under the free first Reidemeister move (shown below).
\end{theorem}

\begin{proof}
Imagine two knot shapes, $S_1$ and $S_2$, which are the same except that one arc of $S_1$ is replaced by a strand containing a loop of length one in $S_2$.  For any assignment of crossings to $S_1$, the resulting knot is the same as the two knots which result from assigning crossings in the same way to $S_2$ and assigning the extra crossing, in the loop of length one, in either possible way.  As $S_2$ has twice as many knots in its family as $S_1$, the resulting knot percentages are the same.
\end{proof}

\begin{center}
\begin{tikzpicture}
\begin{scope} [xshift = -3cm]
\draw (-1.75, 1.5) node {RI};

\draw[thick] (-3,0.75) -- (-3, -0.75);

\draw [thick] (-2,-.5) .. controls (-1.4, .4) .. (-1,0.45);
\draw [thick] (-2, .5) .. controls (-1.4, -.4) .. (-1, -0.45);

\draw [thick, <->] (-2.75,0) -- (-2.25,0);

\draw [thick] (-1, -0.45) .. controls (-0.6, -.45) .. (-.575, 0);
\draw [thick] (-1, 0.45) .. controls (-0.6, .45) .. (-.575, 0);
\end{scope}

\draw (0, 1.5) node {RIII};

\draw [thick] (-2.5,-1) -- (-.5,1);
\draw [thick] (-.5,-1) -- (-2.5,1);

\draw [thick, <->]  (-0.25, 0) -- (.25,0);

\draw [thick] (0.5,-1) -- (2.5,1); 
\draw [thick] (0.5,1) -- (2.5, -1);

\draw [thick] (-1.575, 1) -- (-1.98, 0.63);
\draw [thick] (-2.15, 0.5) .. controls (-2.75, 0) .. (-2.15, -0.5);
\draw [thick] (-1.575, -1) -- (-1.98, -0.63);

\draw [thick] (1.575, 1) -- (1.98, 0.63);
\draw [thick] (2.15, 0.5) .. controls (2.75, 0) .. (2.15, -0.5);
\draw [thick] (1.575, -1) -- (1.98, -0.63);

\end{tikzpicture}
\end{center}

\begin{theorem}
Resultant knot probability, for a mixed knot diagram, is invariant under the above version of the third Reidemeister move.\end{theorem}

\begin{proof}
This version of Reidemeister III is similarly allowed because it can be applied no matter the flavor of the center crossing. The move requires the two off-center crossings to be compatible, which is not guaranteed if either of these crossings are free, so they must be assigned first. (This reasoning also explains why there is no mixed Reidemeister II relation.) 
\end{proof}

In the following theorems, these results can create essentially trivial counterexamples. For example when discussing upper bounds, the trefoil, whose unknot probability is $.75$, could have $n-3$ loops of length one introduced, thus creating a diagram with $n$ crossings and unknot probability $.75$. To prevent these problems, we define the following:

\begin{definition}
A \emph{minimal} free knot diagram of $n$ crossings is a free knot diagram which contains no loops of length one.
\end{definition}

\begin{theorem}
Let $K$ be a minimal free knot diagram with three or more crossings. Then $K$ has some assignment of crossings making the trefoil knot.  
\end{theorem}

\begin{proof}
Suppose $K$ has a loop of length three.  Zoom in on this loop (drawn in red), which has one of two possible forms:

\begin{center}
\begin{tikzpicture}
\begin{scope}
	\draw [dashed] (0,0) circle (1cm);
	\draw [thick] (0,1) .. controls (-.2,.5) .. (-.5,0);
	\draw [thick, red] (-.5,0) .. controls (-.7,-.5) and (-.3,-.5).. (0,-.5) arc (-90:90:.5cm) .. controls (-.3,.5) and (-.7,.5) .. (-.5,0);
	\draw [thick] (-.5,0) .. controls (-.2,-.5) .. (0,-1);
	\node at (0,1.2) {$A$};
	\node at (0,-1.2) {$B$};
\end{scope} 
\begin{scope}[xshift = 3 cm]
	\draw [dashed] (0,0) circle (1cm);
	\draw [thick] (-.72,.72) .. controls (-.5,0) .. (-.5,0);
	\draw [thick, red] (-.5,0) .. controls  (-.3,-.4) and (-.2,-.5) .. (.2,-.5) arc (-90:90:.5cm) .. controls (-.2,.5) and (-.3,.4) .. (-.5,0);
	\draw [thick] (-.5,0) .. controls (-.5,0) ..  (-.72, -.72);
	\draw [thick, blue] (.72,.72) .. controls (.3,0) .. (.72, -.72);
	\node at (.92,.92) {$A$};
	\node at (.92,-.92) {$B$};
	\node at (-.92,-.92) {$C$};
	\node at (-.92,.92) {$D$};
\end{scope}
\end{tikzpicture}
\end{center}

In the first case, turn $K$ into a mixed knot by assigning the crossings that connect $A$ to $B$ outside of the highlighted disk, in a way that makes that outer loop unknotted:  Make $A$ the high point, decreasing monotonically to the low point $B$.  At any self-crossing, assign the higher strand to go over the lower strand.

Once $A$ and $B$ are externally trivially connected, it is straightforward to assign crossings to get the trefoil knot:  traveling within the disk along the strand from $A$, assign the crossings so the current strand goes over, then under, then over.  The result is a trefoil knot.

In the second case, there are two sub-cases to consider.  Outside of the highlighted disk, $A$ can either connect to $D$ or $C$.  (If $A$ connected to $B$, we would have a link, not a knot.)  

If $A$ connects to $D$, and $B$ connects to $C$, then we unknot the external strands from each other and themselves, by making the $A$-$D$ and $B$-$C$ strands downramps, and assigning blended crossings so that $A$-$D$ always goes above $B$-$C$.   First, assign any crossings of the $A$-$D$ and $B$-$C$ strands so that the $A$-$D$ strand is always above the $B$-$C$ strand.  Now, make $A$ and $B$ the highs, and $C$ and $D$ the lows, of their individual external strands, and at any self-crossing, use the height to determine which strand goes over which.

Similarly, if $A$ connects to $C$ and $B$ connects to $D$, make the $A$-$C$ and $B$-$D$ strands downramps, and assign the blended crossings so that $A$-$C$ always goes above $B$-$D$.  

Now, it is straightforward to separate and untangle the external strands, so that the knot looks like one of these:

\begin{center}
\begin{tikzpicture}
	\draw [dashed] (0,0) circle (1cm);
	\draw [thick] (-.72,.72) .. controls (-.5,0) .. (-.5,0);
	\draw [thick] (-.5,0) .. controls  (-.3,-.4) and (-.2,-.5) .. (.2,-.5) arc (-90:90:.5cm) .. controls (-.2,.5) and (-.3,.4) .. (-.5,0);
	\draw [thick] (-.5,0) .. controls (-.5,0) ..  (-.72, -.72);
	\draw [thick] (.72,.72) .. controls (.3,0) .. (.72, -.72);
	\draw [thick] (.72,.72) .. controls (1.1,1.3) and (-1.1,1.3) .. (-.72, .72);
	\draw [thick] (.72,-.72) .. controls (1.1,-1.3) and (-1.1,-1.3) .. (-.72, -.72);
\end{tikzpicture}
 \quad \quad \quad
\begin{tikzpicture}
	\draw [dashed] (0,0) circle (1cm);
	\draw [thick] (-.72,.72) .. controls (-.5,0) .. (-.5,0);
	\draw [thick] (-.5,0) .. controls  (-.3,-.4) and (-.2,-.5) .. (.2,-.5) arc (-90:90:.5cm) .. controls (-.2,.5) and (-.3,.4) .. (-.5,0);
	\draw [thick] (-.5,0) .. controls (-.5,0) ..  (-.72, -.72);
	\draw [thick] (.72,.72) .. controls (.3,0) .. (.72, -.72);
	\draw [thick] (.72,.72) .. controls (0.8,1.0) .. (60:1.3cm) arc (60:220:1.3cm) .. controls  (-0.8,-1.0) .. (-.72, -.72);
	\draw [thick] (.72,-.72)  .. controls (0.9,-1.1) .. (-60:1.5cm) arc (-60:-220:1.5cm) .. controls  (-0.9,1.1) .. (-.72, .72);
\end{tikzpicture}
\end{center}
Both are easily made into the trefoil:  For the first one, starting at 12 o'clock on the external strand and travelling clockwise, make that strand pass over, then under, then over the other strands it encounters.  For the second one, again starting at 12 o'clock on the external strand and travelling clockwise, make that strand pass over, then under, then under, then over the other strands it encounters.

Similarly, in the sub-case where $A$ connects to $C$ and $B$ connects to $D$, we unknot the two strands from each other and then eliminate any self-crossings until we again reach the above base structure that can easily be transformed into the trefoil.

Now suppose $K$ only has loops of length greater than 3. 
By induction on the length of a loop, we will show that any knot has a resolution that is the trefoil.

\begin{center}
\begin{tikzpicture}
	\draw [dashed] (0,0) circle (1cm);
	\draw [thick] (-.72,.72) .. controls (-.5,0) .. (-.5,0);
	\draw [thick, red] (-.5,0) .. controls  (-.3,-.4) and (-.2,-.5) .. (.2,-.5) arc (-90:90:.5cm) .. controls (-.2,.5) and (-.3,.4) .. (-.5,0);
	\draw [thick] (-.5,0) .. controls (-.5,0) ..  (-.72, -.72);
	\draw [thick, green] (.72,.72) .. controls (.5, .5) .. (.45, .25);
	\draw [thick, green] (.72,-.72) .. controls (.5, -.5) .. (.45, -.25);
	
	\draw [thick, green] (.4,.9) .. controls (.25, .7) ..  (.2, .4);
	\draw [thick, green] (.4,-.9) .. controls (.25,-.7) .. (.2, -.4);
	
	\draw [thick, green] (-.4,.9) .. controls (-.2,.6) .. (-.1, .3);
	\draw [thick, green] (-.4,-.9) .. controls (-.2,-.6) .. (-.1, -.3);
	\draw [semithick, dashed] (0.15,0) circle (0.4 cm);
	\node at (-.92,-.92) {$A$};
	\node at (-.92,.92) {$B$};
\end{tikzpicture}
\end{center}

The inductive step is to assign crossings and perform isotopies to shorten the length of a loop by two.
Do this by picking any green strand at the boundary, declaring that location the high point of a downramp which goes through the highlighted disk, and travelling along that strand, assigning every free crossing encountered to have the chosen green strand go over the other strand, until the strand leaves the disk.
The strand thus assigned passes above all other strands on the interior of the disk.  
Thus, by a combination of Reidemeister II and mixed Reidemeister III moves, it can be isotoped away from the loop, shortening the length of the loop by two.

In this way, any free knot diagram with a loop of length more than three has some mixed knot in its family with a loop of length three, thus producing a trefoil knot.
\end{proof}

It is worth noting that a careful analysis of the number of trefoils produced in this way would, indirectly, give an upper bound on the unknot probability. 

In the following results, we consider connected sums of free knot diagrams.  As connected sum is an operation on knots, and its well-definedness on knot diagrams relies on the Reidemeister moves, we shouldn't expect it to be well-defined on free knots.   However, if we are concerned with which knots arise as assignments of a given free knot, then there is no problem:  once crossings are assigned, Reidemeister moves are allowed.  Thus, in Theorem \ref{knotsum1} through Consequence \ref{knotsum2}, we simply refer to a connected sum without worrying about which one is meant.  

\begin{theorem}\label{knotsum1}
Suppose $K_1$ and $K_2$ are knots with shapes $S_1$ and $S_2$ that are components of a connected sum $S_1 \# S_2$. Let $K_2'$ be the mirror image of $K_2$. Then the resultant knot probability of $S_1 \# S_2$ is the same as of a connected sum $S_1 \# S_2'$.
\end{theorem}

\begin{proof}
The difference between $K_2$ and $K_2'$ is that every crossing in $K_2$ has the opposite flavor of the corresponding crossing (which has the exact same connections to all other crossings) in $K_2'$.  Since the starting point of the randomization is the crossingless free knot diagram, the information of what is over and what is under is destroyed. Both connected sums form the same free knot diagram, thus their resultant knot probability is now trivially the same.
\end{proof}

\begin{theorem} \label{knotsumunk}
Suppose $K_1$ and $K_2$ are knots with shapes $S_1$ and $S_2$ that are components of a connected sum $S_1 \# S_2$. Then the probability of getting an unknot resultant $U$ from $S_1 \# S_2$ is $P(S_1\rightarrow U)P(S_2 \rightarrow U)$.
\end{theorem}

\begin{proof}
In the process of determining the resultants of $S_1 \# S_2$, we can shrink one of the components -- without loss of generality $S_1$ -- to be as small as possible, so that $S_2$ is simplified without changing $S_1$ (think of $S_1$ as a bead that may move around $S_1$ as it isotopes). The unknot is a prime knot and can only be created by the knotsum of two unknots, thus $S_1$ and $S_2$ must be simultaneously, and independently, assigned to produce the unknot. Then the resultant probability is the product of the individual resultant unknot probabilities for each of the components.
\end{proof}

\begin{corollary}
Suppose $K$ is a knot with a shape $S$ that is a component of distinct connected sums with the shapes of the trefoil and the figure eight. Then $P(S \# 3_1 \rightarrow U) = P(S \# 4_1 \rightarrow U)$.
\end{corollary}

\begin{proof}
The knot shape $S$ has some probability of resulting in the unknot $P(S \rightarrow U)$ while the trefoil and figure eight have the same known probability $P(3_1 \rightarrow U)=\frac{6}{8}=0.75$ and $P(4_1 \rightarrow U) = \frac{12}{16} = 0.75$. Then from Theorem \ref{knotsumunk}, the probability of getting the unknot as a resultant from either connected sum is the same.
\end{proof}

\begin{theorem} \label{knotsumpr}
Suppose $K_1$ and $K_2$ are knots with shapes $S_1$ and $S_2$ that are components of a connected sum $S_1 \# S_2$. Then the probability of getting a resultant nontrivial prime $K_3$ from $S_1 \# S_2$ is $P(S_1\rightarrow U)P(S_2 \rightarrow K_3)+P(S_1 \rightarrow K_3)P(S_2 \rightarrow U)$.
\end{theorem}

\begin{proof}
If we assign the crossings of $S_2$ (after applying the shrinking procedure as in Theorem \ref{knotsumunk}) such that it becomes the unknot, then we are left with the free knot diagram of $S_1$, from which we can produce $K_3$. Thus, the probability of $K_3$ is at least $P(S_1 \rightarrow U)P(S_2  \rightarrow K_3)+P(S_1 \rightarrow K_3)P(S_2 \rightarrow U)$.

Further $K_3$ knots could only arise if some connected sum of $K_1$ and $K_2$ or any of their resultants could create $K_3$. The assumption that $K_3$ is prime makes this impossible, so the earlier sum calculates the resultant knot probability for $K_3$.
\end{proof}

\begin{theorem}
Suppose $K_1$ and $K_2$ are prime knots with shapes $S_1$ and $S_2$ that are components of a connected sum $S_1 \# S_2$. Then the probability of getting a resultant composite knot $K_3 \# K_4$ from $S_1 \# S_2$ is $P(S_1\rightarrow U)P(S_2 \rightarrow K_3 \# K_4)+P(S_1 \rightarrow K_3 \# K_4)P(S_2 \rightarrow U)+P(S_1 \rightarrow K_3)P(S_2 \rightarrow K_4)+P(S_1 \rightarrow K_4)P(S_2 \rightarrow K_3)$.
\end{theorem}

\begin{proof}
If the composite $K_3 \# K_4$ is a resultant of either of the components, the composition of the unknot and the desired resultant will produce the resultant, reproducing the same two terms from the above theorem. 

Since the resultant in question is itself a knotsum of two knots, then having each of the free components assigned to be the initial knot components of the knotsum will again produce $K_3 \# K_4$. This can be done by one of two ways, with a total probability $P(S_1 \rightarrow K_3)P(S_2 \rightarrow K_4)+P(S_1 \rightarrow K_4)P(S_2 \rightarrow K_3)$.
\end{proof}

\begin{definition}
A \emph{recursive sum $S^N$} of a knot $K$ with shape $S$ is a connected sum of $N$ copies of $S$. 
\end{definition}

\begin{theorem}
Let $K_1$ be a prime knot with a shape $S$. Then the resultant knot probability for a nontrivial prime $K_2$, $P(S^{N+1} \rightarrow K_2)$, is $P(S \rightarrow K_2)P(S \rightarrow U)^N + P(S \rightarrow U)P(S^N \rightarrow K_2)$.
\end{theorem}

\begin{proof}
Given $S^{N+1}=S^N \# S$, 
we apply the formula from Theorem \ref{knotsumpr} and get $P(S^N \rightarrow U)P(S \rightarrow K_2) + P(S^N \rightarrow K_2)P(S \rightarrow U)$.  Applying Theorem \ref{knotsumunk} $N$ times shows $P(S^N \rightarrow U)=P(S \rightarrow U)^N$.
\end{proof}

This motivates the definition, as we get a recursive formula for the resultant knot probability. The appropriate initial condition, for a generic shape $S$ and resultant $K$, is $P(S^0 \rightarrow K)=0$, as an unknotted loop, the identity of the connect sum operation, has no nontrivial resultants. This recovers the prime knot resultant probabilities for any resultant $K_2$ as the first connected sum, that of the prime knot $K_1$ with the unknot, has a probability
\begin{align*}
P(S^1 \rightarrow K_1) =& P(S \rightarrow K_2)P(S \rightarrow U)^0 + P(S \rightarrow U)P(S^0 \rightarrow K_2) \\
=& P(S \rightarrow K_2)
\end{align*}

We rewrite the recurrence relation for $P(S^{N+1} \rightarrow K_2)$ as 
$$x_{N+1} = \alpha^N \beta+ \alpha x_N$$
 using $\alpha = P(S \rightarrow U)$, $\beta = P(S \rightarrow K_2)$, and the previous step in the map $P(S^{N} \rightarrow K_2) = x_{N}$.

Whether $x_{N+1}$ is greater or less than $x_N$ depends on if the first term is large enough to make up for the value lost by the multiplication of $x_N$ by $\alpha<1$, $(1-\alpha)x_N$. Then the system increases in value if $\alpha^N \beta > (1-\alpha)x_N$, and decreases when $\alpha^N \beta < (1-\alpha)x_N$. Indeed, if there is no movement in the system, $x_{N+1} = x_N$, so we again get
$$ x_{N+1} = x_N = \alpha^N \beta+ \alpha x_N \implies x_N = \frac{\alpha^N \beta}{1-\alpha} .$$
 Thus, the behavior of the recurrence relation at some $N$th step is dependent on the value 
$$ \frac{P(S \rightarrow U)^N P(S \rightarrow K_2)}{1-P(S \rightarrow U)} .$$

\begin{theorem} \label{strict}
Let $K_1$ be a prime knot with  shape $S$ such that a prime $K_2$ is a resultant ($\beta \not = 0$). Then the sequence of resultant knot probability of $K_2$ from the recursive sum of $S$, $P(S^{N} \rightarrow K_2)$, strictly increases, then has a maximum value at two or fewer steps after which the probability strictly decreases for every additional connected sum.  
\end{theorem}

\begin{proof}
Suppose after some first $N-1$ recursive sums of $S$, $P(S^{N} \rightarrow K_2)=P(S^{N+1} \rightarrow K_2)$. Then $x_{N+1} = x_{N} \implies (1-\alpha)x_{N+1}=(1-\alpha)x_n = \alpha^N \beta$. Since $\alpha<1$, $\alpha^{N+1} \beta < \alpha^N \beta = (1-\alpha)x_{N+1}$, so $x_{N+2} < x_{N+1}$.

Now suppose after some $N$ steps, $x_N > x_{N+1}$. Then $(1-\alpha)x_{N+1}=(1-\alpha)\left(\alpha^N \beta + \alpha x_N \right)$ and $\alpha^N \beta < (1-\alpha)x_N$. 

\begin{align*}
(1-\alpha)x_{N+1} &= (1-\alpha)\left(\alpha^N \beta + \alpha x_N \right) \\
&= (1-\alpha)\alpha^N \beta + \alpha (1-\alpha)x_N \\
&> (1-\alpha)\alpha^N \beta + \alpha (\alpha^N \beta) \\
&= \alpha^N \beta > \alpha^{N+1} \beta
\end{align*}

\noindent Thus $x_{N+1} > x_{N+2}$. 

Consequently, $x_{N+1}>x_N$ forces $x_{N} > x_{N-1}$.
\end{proof}

\begin{corollary} \label{firstMin}
Let $K_1$ be a prime knot with shape $S$. Then the resultant knot probability of $K_2$ from the recursive sum of $S$, $P(S^{N} \rightarrow K_2)$ has a maximum value at $N=1$ iff $P(S \rightarrow U) \leq 0.5$.
\end{corollary}

\begin{proof}  Using the notation $\alpha$, $\beta$, $x_N$  as above, we have 
$$\alpha \leq \frac{1}{2} \iff 1 \geq \frac{\alpha}{1-\alpha} \iff \beta  \geq \frac{\alpha^1 \beta}{1-\alpha}.$$
As $x_1 = \beta$ and $x_2 = \alpha^1 \beta + \alpha x _1 = 2 \alpha \beta$, this shows that $\alpha \leq \frac{1}{2}$ is equivalent to $x_2 \leq x_1$. According to Theorem \ref{strict}, $x_1$ is then a maximum.
\end{proof}

Assuming our upper bounds hold over the prime knots, Corollary \ref{firstMin} forces recursive sums whose base knot have a resultant unknot percentage less than or equal to one half to similarly respect our proposed absolute bounds over the primes for the trefoil and figure eight knots. 

\begin{theorem}
Let $K_1$ be a nontrivial prime knot with shape $S$. Then the resultant knot probability of $K_2$ from the recursive sum of $S$, $P(S^{N} \rightarrow K_2)$, has a limit $\displaystyle  \lim_{N \to \infty} P(S^{N} \rightarrow K_2)=0$.
\end{theorem}

\begin{proof}
Let $\displaystyle L = \lim_{N \to \infty} x_{N+1}$. Note nontriviality of $K_1$ forces $0 < \alpha < 1$. Then

$$ L =  \lim_{N \to \infty} \alpha^N \beta +  \lim_{N \to \infty} \alpha x_N = 0 + \alpha L $$

\noindent This can only be true if $L=0$.
\end{proof}

\begin{theorem}
Let $K_1$ be a prime knot with shape $S$. Then for a nontrivial prime $K_2$ the resultant knot probability $P(S^{N} \rightarrow K_2)=N \cdot P(S \rightarrow U)^{N-1} P(S \rightarrow K_2)$.
\end{theorem}

\begin{proof}
By induction on $N$. Note $0 \alpha^{-1} \beta = 0 = x_0$, and $1 \alpha^0 \beta = \beta = x_1$, showing agreement with our recursive formula for $P(S_1^{N} \rightarrow K_2)$ in the base cases.

Now suppose $x_k = k \alpha^{k-1} \beta$ for some $N = 0, 1, \ldots, k$. Then $x_{k+1} = \alpha^{k+1-1} \beta + \alpha x_k$. By the inductive hypothesis, 
$$x_{k+1} = \alpha^k \beta + \alpha (k \alpha^{k-1} \beta) =  \alpha^k \beta + k \alpha^{k} \beta = (k+1) \alpha^{k} \beta$$
and thus the formula holds for all $N$.
\end{proof}

\begin{theorem}
Let $K_1$ be a nontrivial prime knot with shape $S$ with resultant unknot probability $\alpha$. Then the resultant knot probability for a nontrivial prime $K_2$ has a singular maximum at 
$$
\begin{array}{cc}
    \begin{array}{ll}
      N=1, & \text{ if }  \alpha < \frac{1}{2} \\
      N=2, & \text{ if }  \frac{1}{2} < \alpha < \frac{2}{3} \\ 
      N=3, & \text{ if }  \frac{2}{3} < \alpha < \frac{3}{4} \\
    \end{array}
\end{array}
$$
\end{theorem}

\begin{proof}
If the resultant knot probability for $K_2$ has solely one maximum value for the given $S$ after $N$ recursive sums, then its value at the $N$th step is smaller at the $N+1$th step and larger at the $N-1$th step. The differences between each step are $(N+1) \alpha^N \beta - N \alpha^{N-1} \beta < 0$ and $N\alpha^{N-1}\beta - (N-1)\alpha^{N-2}\beta > 0$. Solving each inequality for $\alpha$ places bounds on what the resultant unknot probability can be while having a maximum solely at particular values of $N$:

$$ \frac{N-1}{N}  < \alpha < \frac{N}{N+1}$$

\noindent Choosing low values of $N$ then produce the ranges above.
\end{proof}

Assuming Conjecture \ref{unknotconjecture} holds, its corollary that the absolute maximum resultant unknot probability is 0.75 forces the following consequence of the conjecture.

\begin{consequence} \label{knotsum2}
All resultant knot probabilities of nontrivial prime knots from recursive sums have a maximum in the first four connected sums.
\end{consequence}

\begin{proof}
As forced by Theorem \ref{strict}, the resultant knot probability will attain a maximum at two adjacent steps if the difference between each is equal to zero ($N \geq 1$), or, specifically, if:

$$(N+1) \alpha^N \beta - N \alpha^{N-1} \beta = 0 \text{ or } (N+1)\alpha -N = 0$$

Solving for $N$, the first step where the maximum occurs, as a function of $\alpha$, we get:

$$N(\alpha)=\frac{\alpha}{1-\alpha}$$

Note the derivative of $N(\alpha)$, $\frac{1}{(1-\alpha)^2}$, is positive for all $\alpha$, so the largest final step containing a maximum will occur for the largest possible $\alpha$, which we conjecture to be $\alpha=0.75$. $N(0.75)=3$, so the final step containing a maximum will occur at $N=4$, as the value of the resultant probability stays the same.
\end{proof}

Similar to the definition of minimal free knot diagrams above, this creates counterexamples to our later proposed bounds on resultant trefoil and figure eight probabilities. In those cases, we note that they apply to solely prime knots. These counterexamples, and the system itself, depend more on the unknot probability of the base knot shape, not the base shape's resultant probability. For example, $7_1$ appears to produce the most trefoils at a rate of 32.8125\%, but its recursive sum only produces a higher percentage (approximately 35.89\%) for $(7_1)^2$. Meanwhile the recursive sum of the trefoil, with its absolute maximum resultant unknot probability of 75\% produces conjecture breaking results for $(3_1)^2$ through $(3_1)^6$. 

There appear to be further manipulations we can perform within free knot diagrams which do not change their resultant knot probability.  To describe these, we need the language of tangles.

\begin{definition}
A \emph{tangle} is a portion of a knot or link contained in a circular region such that the region's boundary is crossed by exactly four strands.  Reidemeister moves and planar isotopies are allowed if and only if they do not move the four `anchor points' around the boundary circle.
\end{definition}

A free tangle is a tangle with free crossings instead of true crossings.   In the illustrations below, $T_i$ are any free tangles. 

Observations show that $90\degree$ rotations and reflections across the diagonals of a tangle (i.e., an axis connecting two opposite anchor points) change the resultant knot probability. However,

\begin{conjecture}
$180\degree$ rotations and vertical and horizontal reflections (i.e., reflections which fix none of the four anchor points) of a tangle within a knot projection do not affect resultant knot probability. 
\end{conjecture}

For example, the following free knot diagrams with a variety of choices of $T_i$ (including nonsymmetric possibilities) were seen to have the same resultant knot probabilities.

\begin{center}
\begin{tikzpicture}
\begin{scope}
    \draw (0,0) circle [radius =.5];
    \draw (0,2) circle [radius = .5];
    \draw (1,4) circle [radius = .5];
    \draw (-1, 4) circle [radius = .5];
    
    \draw (0,0) node {$T_4$};
    \draw (0,2) node {$T_3$};
    \draw (1,4) node {$T_2$};
    \draw (-1,4) node {$T_1$};
    
    \draw (-0.7,4.4) .. controls (0, 4.5) .. (0.7, 4.4);
    \draw (-0.6, 3.7) .. controls (0, 3.6) .. (0.6, 3.7);
    \draw (-1.3,3.6) .. controls (-1.3, 0.5) .. (-0.5, 0);
    \draw (1.3,3.6) .. controls (1.3, 0.5) .. (0.5, 0);
    \draw (-0.2, 0.45) .. controls (-0.3, 0.95) .. (-0.2, 1.55);
    \draw (0.2, 0.45) .. controls (0.3, 0.95) .. (0.2, 1.55);
    \draw (-0.8, 3.55) .. controls  (-0.7, 3.05).. (-0.2, 2.45); 
    \draw (0.8, 3.55) .. controls (0.7, 3.05) .. (0.2, 2.45);
\end{scope}

\draw[thick, <->] (1.5,2) -- (2.5,2);

\begin{scope}[xshift = 4cm]
    \draw (0,0) circle [radius =.5];
    \draw (0,2) circle [radius = .5];
    \draw (1,4) circle [radius = .5];
    \draw (-1, 4) circle [radius = .5];
    
    \draw (0,0) node {$T_4$};
    \draw (0,2) node [rotate=180] {$T_3$};
    \draw (1,4) node {$T_2$};
    \draw (-1,4) node {$T_1$};
    
    \draw (-0.7,4.4) .. controls (0, 4.5) .. (0.7, 4.4);
    \draw (-0.6, 3.7) .. controls (0, 3.6) .. (0.6, 3.7);
    \draw (-1.3,3.6) .. controls (-1.3, 0.5) .. (-0.5, 0);
    \draw (1.3,3.6) .. controls (1.3, 0.5) .. (0.5, 0);
    \draw (-0.2, 0.45) .. controls (-0.3, 0.95) .. (-0.2, 1.55);
    \draw (0.2, 0.45) .. controls (0.3, 0.95) .. (0.2, 1.55);
    \draw (-0.8, 3.55) .. controls  (-0.7, 3.05).. (-0.2, 2.45); 
    \draw (0.8, 3.55) .. controls (0.7, 3.05) .. (0.2, 2.45);
\end{scope}
\end{tikzpicture}
\end{center}

\begin{consequence}
Resultant knot probability is invariant under flyping.
\end{consequence}

\begin{definition}
The flype, defined diagrammatically below, rotates a tangle and an adjacent crossing 180 degrees within a knot diagram.
\begin{center}
\begin{tikzpicture}
\draw (1.5,0.25) circle [radius = .5];
\draw (1.5,0.25) node{$T_2$};
\draw (1.94,0.5) -- (2.29,0.5);
\draw (1.94,0)--(2.29,0);
\draw (2.29,0)--(2.79,0.5);
\draw (2.29,0.5)--(2.79,0);
\draw (2.79,0)--(3,0);
\draw (2.79, 0.5)--(3,0.5);
\draw (.67, 0) -- (1.06, 0);
\draw (.66, 0.5) -- (1.06, 0.5);

\draw[thick, <->] (3.25,0.25)--(3.83,0.25);

\begin{scope}[xshift=-1cm]
\draw (6.5,0.25) node [rotate=180] {$T_2$};
\draw (6.5, 0.25) circle [radius = .5];
\draw (6.92,0) -- (7.15, 0);
\draw (6.92, 0.5) -- (7.15, 0.5);
\draw (5.2, 0.5) -- (5.38, 0.5);
\draw (5.38, 0.5) -- (5.88, 0);
\draw (5.2, 0) -- (5.38, 0);
\draw (5.38, 0) -- (5.88, 0.5);
\draw (5.88, 0) -- (6.08,0);
\draw (5.88, 0.5) -- (6.08,.5);
\end{scope}
\end{tikzpicture}
\end{center}
\end{definition}

\begin{proof}
The flype is still a rotation of a tangle as in the above conjecture. This can be seen by defining the tangle to be rotated as the sum of $T_2$ and the free 1 tangle. Thus, the flype preserves the resultant probability.
\end{proof}

\begin{conjecture}
For all free knot diagrams and nontrivial knots, the resultant knot probability is less than $\frac{1}{2}$.
\end{conjecture}

For many of the prime knots through eight crossings, their free knot diagram has a resultant unknot probability greater than $\frac{1}{2}$, which verifies the conjecture for these knots.  As the number of crossings increase, the resultant unknot probability tends to decrease and often drops below $\frac{1}{2}$. However, the larger number of crossings permits more higher crossing resultants (e.g., a nine crossing free knot diagram producing eight crossing resultants) to fill the space that the resultant unknots cede, preventing any one resultant from crossing the $\frac{1}{2}$ probability threshold.

%% file: Observations.tex
\label{tangles}
Tangles, as introduced above, give us a concise notation for a wide variety of knots.

\begin{definition}
Conway's tangle notation, introduced in \cite{MR0258014}, uses the phrase \emph{$\pm n$ tangle} to describe a tangle which begins with two parallel horizontal strands and twists them $n$ times in the same direction.  The sign is dictated by the sign of overstrand's slope. 
\end{definition}

 For example, here is the -5 tangle:

\begin{center}
\begin{tikzpicture}
\draw[dashed] (0, 0) ellipse (3 cm and 1.1 cm); 
\begin{scope}
\draw (.1, .1) -- (.5,.5);
\draw (-.5,.5) -- (.5,.-.5);
\draw (-.5, -.5) -- (-.1,-.1);
\end{scope}
\begin{scope}[xshift=1.25cm]
\draw (.1, .1) -- (.5,.5);
\draw (-.5,.5) -- (.5,.-.5);
\draw (-.5, -.5) -- (-.1,-.1);
\end{scope}
\begin{scope}[xshift=-1.25cm]
\draw (.1, .1) -- (.5,.5);
\draw (-.5,.5) -- (.5,.-.5);
\draw (-.5, -.5) -- (-.1,-.1);
\end{scope}
\begin{scope}[xshift=2.5cm]
\draw (.1, .1) -- (.5,.5);
\draw (-.5,.5) -- (.5,.-.5);
\draw (-.5, -.5) -- (-.1,-.1);
\end{scope}
\begin{scope}[xshift=-2.5cm]
\draw (.1, .1) -- (.5,.5);
\draw (-.5,.5) -- (.5,.-.5);
\draw (-.5, -.5) -- (-.1,-.1);
\end{scope}
 
\draw (.5,.5) .. controls (.625, .6) .. (.75,.5);
\draw (.5, -.5) .. controls (.625, -.6) .. (.75, -.5);
\draw (-.5, .5) .. controls (-.625, .6) .. (-.75, .5);
\draw (-.5, -.5) .. controls (-.625, -.6) .. (-.75, -.5);
\draw (1.75, .5) .. controls (1.875, .6) .. (2, .5);
\draw (1.75, -.5) .. controls (1.875, -.6) .. (2, -.5);
\draw (-1.75, .5) .. controls (-1.875, .6) .. (-2, .5);
\draw (-1.75, -.5) .. controls (-1.875, -.6) .. (-2, -.5);
\end{tikzpicture}
\end{center}

There are two trivial crossingless tangles, the zero and $\infty$ tangles:

\begin{center}
\begin{tikzpicture}
\draw [dashed] (-2, 0) circle (0.75 cm);
\draw (-2.75, .75) .. controls (-2, 0) .. (-1.25, .75);
\draw (-2.75, -.75) .. controls (-2, 0) .. (-1.25, -.75);
\draw (-2, -1.25) node {0};

\draw[dashed] (0,0) circle (0.75 cm);
\draw (.75, .75) .. controls (0,0) .. (.75, -.75);
\draw (-.75, .75) .. controls (0,0) .. (-.75, -.75);
\draw (0, -1.25) node {$\infty$};
\end{tikzpicture}
\end{center}

Quite complex knots can quickly be formed by combining tangles in a variety of ways.  The class of rational tangles are the tangles formed by the operation of multiplication on  base tangles $T_1$ and $T_2$. This requires $T_1$ to be reflected across a slope -1 diagonal through the tangle before connecting the adjacent unattached strands of both. 

$$
\begin{tikzpicture}
\draw[dashed] (-3,0) circle (0.75cm);
\draw (-2.75, .25) node {$T_1$};
\draw (-3.5, .575) -- (-3.75, 0.825);
\draw (-3.5, -.575) -- (-3.75, -0.825);
\draw (-2.5, .575) -- (-2.25, 0.825);
\draw (-2.5, -.575) -- (-2.25, -0.825);
\draw[dashed, red] (-3.5, .575) -- (-2.5, -.575);

\draw[->] (-2,0) -- (-1.5,0);

\begin{scope}[xshift = 2.5cm]
\draw[dashed] (-3,0) circle (0.75cm);
\draw (-3, 0) node [xscale = -1, yscale= 1, rotate = 90]  {$T_1$};
\draw (-3.5, .575) -- (-3.75, 0.825);
\draw (-3.5, -.575) -- (-3.75, -0.825);
\draw (-2.5, .575) -- (-2.25, 0.825);
\draw (-2.5, -.575) -- (-2.25, -0.825);
\end{scope}

\draw[dashed] (2,0) circle (0.75cm);
\draw (2,0) node [xscale = -1, yscale= 1, rotate = 90]  {$T_1$};
\draw[dashed] (4,0) circle (0.75cm);
\draw (4, 0) node {$T_2$};
\draw (1.5, .575) -- (1.25, 0.825);
\draw (1.5, -.575) -- (1.25, -0.825);
\draw (4.5, .575) -- (4.75, 0.825);
\draw (4.5, -.575) -- (4.75, -0.825);
\draw (2.5, .575) .. controls (3, 0.9) .. (3.5, .575);
\draw (2.5, -.575) .. controls (3, -0.9) .. (3.5, -.575);

\draw (5.9, 0) node {= \, $T_1 \cdot T_2$};

\end{tikzpicture}
$$

Conway's notation writes the multiplication of tangles as a space separated list (for example, 3 -1 2) of the individual tangle values. From a rational tangle we can produce a continued fraction. The tangle 3 -1 2 has continued fraction
$$ 2 + \frac{1}{-1 + \frac{1}{3}} = 2 - \frac{3}{2} = \frac{1}{2}$$

\begin{theorem}  \cite{MR0258014} 
Two rational tangles are isotopic to each other if their continued fraction is equivalent.
\end{theorem}

A tangle can be made into a knot or link by taking its closure, which connects the upper and lower two pairs of exterior arcs.

\begin{center}
\begin{tikzpicture}
\draw[dashed] (0,0) circle (0.75 cm);
\draw (0,0) node {$T$};
\draw (0.5, .575) -- (0.75, 0.825);
\draw (0.5, -.575) -- (0.75, -0.825);
\draw (-0.5, .575) -- (-0.75, 0.825);
\draw (-0.5, -.575) -- (-0.75, -0.825);
\draw (-0.75, .825) .. controls (-1, 1.075) .. (0, 1.125); 
\draw (0.75, .825) .. controls (1, 1.075) .. (0, 1.125); 
\draw (-0.75, -.825) .. controls (-1, -1.075) .. (0, -1.125); 
\draw (0.75, -.825) .. controls (1, -1.075) .. (0, -1.125); 
\end{tikzpicture}
\end{center}

The continued fraction still has use in differentiating between knots once we move to closures of tangles by the following process:

\begin{lemma}\label{KnotFracEquiv}(From
 \cite{MR82104}, as quoted in \cite{MR1953344}) Suppose there exist two rational tangles with continued fractions $\frac{p}{q}$ and $\frac{p'}{q'}$, where p, q and p', q' are relatively prime. Let K$\left(\frac{p}{q}\right)$ and K$\left(\frac{p'}{q'}\right)$ be the knots formed by the closure of the respective rational tangles. Then K$\left(\frac{p}{q}\right)$ and K$\left(\frac{p'}{q'}\right)$ are equivalent (up to isotopy) iff:
\begin{itemize}
\item p = p'
\item either q $\equiv$ q' mod p or qq' $\equiv$ 1 mod p.
\end{itemize}
\end{lemma}

Another, simpler operation on tangles is addition, which connects the adjacent strands like in the second step of tangle multiplication. Tangles constructed using both addition and multiplication are called algebraic tangles. 

\begin{center}
\begin{tikzpicture}
\draw[dashed] (-3,0) circle (0.75cm);
\draw[dashed] (-1,0) circle (0.75cm);
\draw (-3,0) node {$T_1$};
\draw (-1,0) node {$T_2$};
\draw (-2, 1.25) node {$T_1 + T_2$};

\draw (-0.5,.575) -- (-0.25,0.825);
\draw (-0.5,-.575) -- (-0.25,-0.825);
\draw (-3.5, .575) -- (-3.75, 0.825);
\draw (-3.5, -.575) -- (-3.75, -0.825);
\draw (-2.5, .575) .. controls (-2, 0.9) .. (-1.5, .575);
\draw (-2.5, -.575) .. controls (-2, -0.9) .. (-1.5, -.575);
\end{tikzpicture}
\end{center}

\begin{definition}
The foil knots, or n-foils, are the closures of $n$ tangles, where $n$ is odd. These knots include the trefoil, pentafoil, $7_1$, $9_1$, and the unknot (with the first Reidemeister relation applied).
\end{definition}

\begin{center}
\begin{tikzpicture}
\draw (.25,0) -- (0,0);
\draw (.25,.5) -- (0,.5);
\draw (.25, 0) -- (.75, .5);
\draw (.25,.5) -- (.75,0);
\draw (.75,0)--(1,0);
\draw (.75, .5) -- (1,.5);

\node[draw=none] (ellipsis1) at (1.425,0.25) {$\cdots$};

\draw (1.75,0) -- (2,0);
\draw (1.75,.5) -- (2,.5);
\draw (2, 0) -- (2.5, .5);
\draw (2,.5) -- (2.5,0);
\draw (2.5,0)--(2.75,0);
\draw (2.5, .5) -- (2.75,.5);

\draw (2.75,0) arc [start angle = 270, end angle = 90, radius = -1mm];
\draw (0,0) arc [start angle = 90, end angle = 270, radius = 1mm];
\draw (2.75,.5) arc [start angle = 90, end angle = 270, radius = -1mm];
\draw (0,.5) arc [start angle = 270, end angle = 90, radius = 1mm];
\draw (0,0.7) -- (2.75, 0.7);
\draw (0,-0.2) -- (2.75, -0.2);

\draw [decorate,decoration={brace,amplitude=10pt},xshift=0pt,yshift=-4pt]  (2.5,-0.2)--(0.25,-0.2) node [black,midway,yshift=-0.6cm]{$n$};
\end{tikzpicture}
\end{center}

\noindent An odd number of crossings is required since an even number of crossings would require two components, making the result a link.

\begin{theorem}\label{Foils}
The free $n$-foil produces ${n} \choose {\frac{n-k}{2}}$ left $k$-foils and ${n} \choose {\frac{n-k}{2}}$ right $k$-foils.
\end{theorem}

\begin{proof}
Previously, we have counted a knot and its reflection in the same category, even if it is chiral. Here, we count a knot and its reflection as distinct knots. 

A foil knot can have two types of crossings, ``positive slope" and ``negative slope" crossings, where the referenced slope is that of the overstrand:

\begin{center}
\begin{tikzpicture}
\begin{scope}[scale=0.75]
\draw  (-.15, .15) -- (-1,1);
\draw (-1,-1) -- (1,1);
\draw (.15, -.15) -- (1,-1);

\draw  (4.5, -1) -- (2.5,1);
\draw  (3.65,.15) -- (4.5,1);
\draw (3.35, -.15) -- (2.5,-1);
\end{scope}
\end{tikzpicture}
\end{center}

Given a free $n$-foil, we may choose $\ell$ of its crossings to be positive slope crossings and the remaining $n-\ell$ to be negative slope crossings.  Then, if $0 < \ell < n$, we may find a positive crossing next to a negative crossing, and remove the pair via Reidemeister II.  Proceeding thus until all crossings of one type have been removed, we are left with either $n-2 \ell$ or $2 \ell -n$ crossings of a single type.  Substituting $\ell = (n-k)/2$ or $\ell=(n+k)/2$ produces the formula stated above.

\end{proof} 

We illustrate the above in the case of the free $5$-foil (all these drawings should be understood to connect the uppermost two corners to each other, and the lowermost two to each other).  If we choose zero positive slope crossings, we get the pentafoil:

\begin{center}
\begin{tikzpicture}
\foreach \x in {0,1,2,3,4} {
    \begin{scope}[xshift = \x cm]
        \draw (.25,0) -- (0,0);
        \draw (.25,.5) -- (0,.5);
        \draw (.25,0) -- (.4,.15);
        \draw (.25, .5) -- (.75, 0);
        \draw (.6, .35) -- (.75, .5);
        \draw (.75,0) -- (1,0);
        \draw (.75, .5) -- (1,.5);
    \end{scope}};
\end{tikzpicture}
\end{center}
There are ${5 \choose 1} = 5$ ways to choose one positive slope crossing.  One such diagram:
\begin{center}
\begin{tikzpicture}
\foreach \x in {0,1,3,4} {
    \begin{scope}[xshift = \x cm]
        \draw (.25,0) -- (0,0);
        \draw (.25,.5) -- (0,.5);
        \draw (.25,0) -- (.4,.15);
        \draw (.25, .5) -- (.75, 0);
        \draw (.6, .35) -- (.75, .5);
        \draw (.75,0) -- (1,0);
        \draw (.75, .5) -- (1,.5);
    \end{scope}};
\foreach \x in {2} {
    \begin{scope} [xshift = \x cm]
    	\draw (.25,0) -- (0,0);
    	\draw (.25,.5) -- (0,.5);
    	\draw (.25, 0) -- (.75, .5);
    	\draw (.25, .5) -- (.4, .35);
    	\draw (.6,.15) -- (.75,0);
    	\draw (.75,0)--(1,0);
    	\draw (.75, .5) -- (1,.5);
    \end{scope}};
\end{tikzpicture}
\end{center}
Now we can use Reidemeister II to remove two crossings, and produce a trefoil:
\begin{center}
\begin{tikzpicture}
\foreach \x in {0,1,3,4} {
    \begin{scope}[xshift = \x cm]
        \draw (.25,0) -- (0,0);
        \draw (.25,.5) -- (0,.5);
        \draw (.25,0) -- (.4,.15);
        \draw (.25, .5) -- (.75, 0);
        \draw (.6, .35) -- (.75, .5);
        \draw (.75,0) -- (1,0);
        \draw (.75, .5) -- (1,.5);
    \end{scope}};
\foreach \x in {2} {
    \begin{scope} [xshift = \x cm]
    	\draw (.25,0) -- (0,0);
    	\draw (.25,.5) -- (0,.5);
    	\draw (.25, 0) -- (.75, .5);
    	\draw (.25, .5) -- (.4, .35);
    	\draw (.6,.15) -- (.75,0);
    	\draw (.75,0)--(1,0);
    	\draw (.75, .5) -- (1,.5);
    \end{scope}};
\draw[red, dashed] (3, 0.25) ellipse (1 cm and 0.5 cm); 

\draw[->] (5.25, 0.25) -- (5.75,0.25);

\foreach \x in {6,7,8} {
    \begin{scope}[xshift = \x cm]
        \draw (.25,0) -- (0,0);
        \draw (.25,.5) -- (0,.5);
        \draw (.25,0) -- (.4,.15);
        \draw (.25, .5) -- (.75, 0);
        \draw (.6, .35) -- (.75, .5);
        \draw (.75,0) -- (1,0);
        \draw (.75, .5) -- (1,.5);
    \end{scope}};
\end{tikzpicture}
\end{center}
There are ${5 \choose 2}=10$ ways to choose two crossings to have positive slope.
All of these produce the unknot:
\begin{center}
\begin{tikzpicture}
\foreach \x in {0,1,4} {
    \begin{scope}[xshift = \x cm]
        \draw (.25,0) -- (0,0);
        \draw (.25,.5) -- (0,.5);
        \draw (.25,0) -- (.4,.15);
        \draw (.25, .5) -- (.75, 0);
        \draw (.6, .35) -- (.75, .5);
        \draw (.75,0) -- (1,0);
        \draw (.75, .5) -- (1,.5);
    \end{scope}};
\foreach \x in {2,3} {
    \begin{scope} [xshift = \x cm]
    	\draw (.25,0) -- (0,0);
    	\draw (.25,.5) -- (0,.5);
    	\draw (.25, 0) -- (.75, .5);
    	\draw (.25, .5) -- (.4, .35);
    	\draw (.6,.15) -- (.75,0);
    	\draw (.75,0)--(1,0);
    	\draw (.75, .5) -- (1,.5);
    \end{scope}};
\draw[red, dashed] (4, 0.25) ellipse (1 cm and 0.5 cm); 
\draw[blue, dashed] (2, 0.25) ellipse (1 cm and 0.5 cm); 

\draw[->] (5.25, 0.25) -- (5.75,0.25);

\foreach \x in {6} {
    \begin{scope}[xshift = \x cm]
        \draw (.25,0) -- (0,0);
        \draw (.25,.5) -- (0,.5);
        \draw (.25,0) -- (.4,.15);
        \draw (.25, .5) -- (.75, 0);
        \draw (.6, .35) -- (.75, .5);
        \draw (.75,0) -- (1,0);
        \draw (.75, .5) -- (1,.5);
    \end{scope}};
\end{tikzpicture}
\end{center}
If we choose three positive slope crossings, we can again apply two Reidemeister II moves to recognize that we have an unknot:
\begin{center}
\begin{tikzpicture}
\foreach \x in {1,4} {
    \begin{scope}[xshift = \x cm]
        \draw (.25,0) -- (0,0);
        \draw (.25,.5) -- (0,.5);
        \draw (.25,0) -- (.4,.15);
        \draw (.25, .5) -- (.75, 0);
        \draw (.6, .35) -- (.75, .5);
        \draw (.75,0) -- (1,0);
        \draw (.75, .5) -- (1,.5);
    \end{scope}};
\foreach \x in {0,2,3} {
    \begin{scope} [xshift = \x cm]
    	\draw (.25,0) -- (0,0);
    	\draw (.25,.5) -- (0,.5);
    	\draw (.25, 0) -- (.75, .5);
    	\draw (.25, .5) -- (.4, .35);
    	\draw (.6,.15) -- (.75,0);
    	\draw (.75,0)--(1,0);
    	\draw (.75, .5) -- (1,.5);
    \end{scope}};
\draw[red, dashed] (4, 0.25) ellipse (1 cm and 0.5 cm); 
\draw[blue, dashed] (1, 0.25) ellipse (1 cm and 0.5 cm); 

\draw[->] (5.25, 0.25) -- (5.75,0.25);

\foreach \x in {6} {
    \begin{scope} [xshift = \x cm]
    	\draw (.25,0) -- (0,0);
    	\draw (.25,.5) -- (0,.5);
    	\draw (.25, 0) -- (.75, .5);
    	\draw (.25, .5) -- (.4, .35);
    	\draw (.6,.15) -- (.75,0);
    	\draw (.75,0)--(1,0);
    	\draw (.75, .5) -- (1,.5);
    \end{scope}};
\end{tikzpicture}
\end{center}
When we choose four positive slope crossings, we again create a trefoil:
\begin{center}
\begin{tikzpicture}
\foreach \x in {2} {
    \begin{scope}[xshift = \x cm]
        \draw (.25,0) -- (0,0);
        \draw (.25,.5) -- (0,.5);
        \draw (.25,0) -- (.4,.15);
        \draw (.25, .5) -- (.75, 0);
        \draw (.6, .35) -- (.75, .5);
        \draw (.75,0) -- (1,0);
        \draw (.75, .5) -- (1,.5);
    \end{scope}};
\foreach \x in {0,1,3,4} {
    \begin{scope} [xshift = \x cm]
    	\draw (.25,0) -- (0,0);
    	\draw (.25,.5) -- (0,.5);
    	\draw (.25, 0) -- (.75, .5);
    	\draw (.25, .5) -- (.4, .35);
    	\draw (.6,.15) -- (.75,0);
    	\draw (.75,0)--(1,0);
    	\draw (.75, .5) -- (1,.5);
    \end{scope}};
    \draw[red, dashed] (2, 0.25) ellipse (1 cm and 0.5 cm); 

\draw[->] (5.25, 0.25) -- (5.75,0.25);

\foreach \x in {6,7,8} {
    \begin{scope} [xshift = \x cm]
    	\draw (.25,0) -- (0,0);
    	\draw (.25,.5) -- (0,.5);
    	\draw (.25, 0) -- (.75, .5);
    	\draw (.25, .5) -- (.4, .35);
    	\draw (.6,.15) -- (.75,0);
    	\draw (.75,0)--(1,0);
    	\draw (.75, .5) -- (1,.5);
    \end{scope}};
\end{tikzpicture}
\end{center}
And in choosing all 5 to be positive slope crossings, we again get a pentafoil.
\begin{center}
\begin{tikzpicture}
\foreach \x in {6,7,8,9,10} {
    \begin{scope} [xshift = \x cm]
    	\draw (.25,0) -- (0,0);
    	\draw (.25,.5) -- (0,.5);
    	\draw (.25, 0) -- (.75, .5);
    	\draw (.25, .5) -- (.4, .35);
    	\draw (.6,.15) -- (.75,0);
    	\draw (.75,0)--(1,0);
    	\draw (.75, .5) -- (1,.5);
    \end{scope}};
\end{tikzpicture}
\end{center}

Since the maximum value of the binomial coefficient occurs when producing unknots, it is easy to see that the unknot will have the highest resultant knot probability for any resultant of a member of the foil family. 

\begin{corollary}
The expected number of crossings for a resultant of an n-foil is $2^{1-n} \sum^n_{k=3} k {n \choose \frac{n-k}{2}}$, where the values of $k$ are odd.
\end{corollary}

\begin{proof}
Since foils produce an equal number of right and left $k$-foils, the probability of getting a resultant with $k$ crossings is $\frac{2}{2^n} {n \choose \frac{n-k}{2}}$. The 1-foil is the unknot, thus its probability has a coefficient of zero in the expectation value. Then the expectation value sum, over the allowed odd crossing values, starts at the trefoil and goes up to the $n$-foil.
\end{proof}

\begin{theorem} \label{TrefLimit}
The limit as the number of crossings $n$ goes to infinity of resultant $k$-foil probability ($k \leq n$, $k$ is odd) is 0 for foil knots.
\end{theorem}

\begin{proof}
For large $m$, Stirling's approximation tells us 
$$ m! \sim \left(\frac{m}{e}\right)^m \sqrt{2 \pi m} $$
Then the binomial coefficient, for large $m$ and $\ell$ is
\begin{align*}
{m \choose \ell} = \frac{m!}{\ell! (m-\ell)!} &\sim \frac{\left(\frac{m}{e}\right)^m \sqrt{2 \pi m}}{\left(\frac{\ell}{e}\right)^\ell \sqrt{2 \pi \ell} \left(\frac{m-\ell}{e}\right)^{m-\ell} \sqrt{2 \pi (m-\ell)}} \\
&= \sqrt{\frac{m}{2 \pi \ell (m-\ell)}} \frac{m^m}{\ell^\ell (m-\ell)^{m-\ell}}
\end{align*}
And in particular,
\begin{align*}
{n \choose \frac{n-k}{2}} &\sim \sqrt{\frac{n}{2 \pi \frac{n-k}{2} \frac{n+k}{2}}} \frac{n^n}{\frac{n-k}{2}^{\frac{n-k}{2}} \frac{n+k}{2}^{\frac{n+k}{2}}} \\
&= \sqrt{\frac{2n}{\pi}} \frac{n-k}{(n+k)^2} \frac{(2n)^n}{(n-k)^{\frac{n}{2}}(n+k)^{\frac{n}{2}}}
\end{align*}

Then in the limit of the resultant probability for $k$-foils, an extra factor of $\frac{2}{2^n}$ leads to:
$$ \frac{2}{2^n} {n \choose \frac{n-k}{2}} \sim \sqrt{\frac{8}{\pi}} \frac{n-k}{(n+k)^2} \frac{n^{n+\frac{1}{2}}}{(n-k)^{\frac{n}{2}}(n+k)^{\frac{n}{2}}} $$
The highest power of $n$ occuring in the denominator is $n+2$, and is larger than the highest power of $n$ in the numerator (which is $n+\frac{3}{2}$).  Thus the resultant $k$-foil probability goes to zero as $n$ goes to infinity.
\end{proof}

\begin{conjecture} \label{TrefoilUpperBound}
Among all minimal prime free knot diagrams with $n$ and $n+1$ crossings, where $n$ is odd and $n \geq 3$, the foil knot with $n$ crossings has the most trefoil descendants.
\end{conjecture}

To create a trefoil, an assignment of crossings must create three alternating crossings aligned in a row. When trying to find assignments that create such a shape, the foil family clearly has an advantage, as it only needs a certain number of crossings isotoped away by the second Reidemeister relation and the free crossings assigned to be alternating. Meanwhile any other knot $K$ with $n$ crossings is not in that form and has some inherent cost in eliminating the complicating crossings that differentiate $K$ from the foils. Based on the calculated knot resultants, we expect this cost is high enough to prohibit the knot from creating more trefoils than the $n$-foil.

\begin{consequence} \label{MaxTref}
For prime knots, the absolute maximum k-foil probability is $2^{1-k^2} {k^2 \choose \frac{k^2-k}{2}}$, and consequently, the maximum trefoil percentage is 32.8125. 
\end{consequence}

\begin{proof}
We create a sequence $a_n$ calculating the resultant trefoil probability for $n$-foils, where $n$ must be incremented by 2.
\begin{equation*}
a_n = \frac{2{n \choose \frac{n-3}{2}}}{2^n} = \frac{1}{2^{n-1}} \frac{n!}{\frac{n-3}{2}! \frac{n+3}{2}!}
\end{equation*}
The sequence's next term is then
\begin{align*}
a_{n+2} &=  \frac{1}{2^{n+1}}  \frac{(n+2)!}{\frac{n-1}{2}! \frac{n+5}{2}!} \\
&= \frac{(n+2)(n+1)}{2^2 \frac{n-1}{2} \frac{n+5}{2} } \frac{1}{2^{n-1}} \frac{n!}{\frac{n-3}{2}! \frac{n+3}{2}!} \\
&= \frac{(n+2)(n+1)}{(n-1)(n+5)} a_n,
\end{align*}
and 
$a_n$ will be monotonically decreasing when

\begin{equation*}
\frac{(n+2)(n+1)}{(n-1)(n+5)} < 1 \text{ or } n>7
\end{equation*}

As for higher crossing members of the foil family, we define the generalized sequence $g_n$ as 
\begin{equation*}
g_n = \frac{2{n \choose \frac{n-k}{2}}}{2^n} =  \frac{1}{2^{n-1}} {n \choose \frac{n-k}{2}}
\end{equation*}
Using the same procedure as the special case of $k=3$ above for all foil knots,
\begin{equation*}
g_{n+2} = \frac{(n+2)(n+1)}{(n+2-k)(n+2+k)} g_n
\end{equation*}
This sequence will be monotonically decreasing when $n > k^2$ and monotonically increasing when $n<k^2-2$. 
Observing that $\displaystyle g_{ k^2-2}=g_{k^2}$, we have two occurences of the  maximum probability for a given k-foil:
\begin{equation*}
2^{3-k^2} {k^2-2 \choose \frac{k^2-k-2}{2}} = 2^{1-k^2} {k^2 \choose \frac{k^2-k}{2}}
\end{equation*}

Then for the trefoil ($k=3$), $7_1$ and $9_1$ have the maximum trefoil percentage of 32.8125.
\end{proof}

Conjecture \ref{TrefoilUpperBound} would also allows us to strengthen Theorem \ref{TrefLimit}'s results by applying it to all knots, forcing the limit of resultant $k$-foil probability as the number of crossings $n$ goes to infinity to be 0.

\begin{conjecture}
Let K be a minimal non-foil (or non-connected-sum-of-solely-foils) free knot diagram with four or more crossings. Then K has some assignment of crossings making the figure eight knot.
\end{conjecture}

Theorem \ref{Foils} confirms that all resultants of foil family knot diagrams are also foils, so no figure eights are produced. These knots appear to be the only ones without a figure eight resultant. This is sensible since they are the only knots with braid index of two.

\begin{definition}
A \emph{braid} on $n$ strands is a collection of non-crossing and monotonically increasing paths through 3-space connecting $n$ lower points to $n$ upper points.
The closure of a braid is the knot or link created by connecting the top points of the braid to the bottom points by going around and outside the braid, as shown below.
\end{definition}

\begin{center}
\begin{tikzpicture}
\draw (0,0) rectangle (1,1);
\draw (0.5,0.5) node{B};
\draw (0.1,0) -- (0.1, -0.5);
\draw (0.1,-0.5) arc [start angle=-180, end angle=0, radius=11mm];
\draw (0.3, 0) -- (0.3, -0.5);
\draw (0.3,-0.5) arc [start angle=-180, end angle=0, radius=9mm];
\draw(0.5,0) -- (0.5,-0.5);
\draw (0.5,-0.5) arc [start angle=-180, end angle=0, radius=7mm];
\draw(0.7,0) -- (0.7,-0.5);
\draw (0.7,-0.5) arc [start angle=-180, end angle=0, radius=5mm];
\draw(0.9,0) -- (0.9,-0.5);
\draw (0.9,-0.5) arc [start angle=-180, end angle=0, radius=3mm];
\draw (0.1,1) -- (0.1, 1.5);
\draw (0.3, 1) -- (0.3, 1.5);
\draw(0.5,1) -- (0.5,1.5);
\draw(0.7,1) -- (0.7,1.5);
\draw(0.9,1) -- (0.9,1.5);
\draw (0.9,1.5) arc [start angle=180, end angle=0, radius=3mm];
\draw (0.7,1.5) arc [start angle=180, end angle=0, radius=5mm];
\draw (0.5,1.5) arc [start angle=180, end angle=0, radius=7mm];
\draw (0.3,1.5) arc [start angle=180, end angle=0, radius=9mm];
\draw (0.1,1.5) arc [start angle=180, end angle=0, radius=11mm];
\draw (1.5,1.5) -- (1.5, -0.5);
\draw (1.7,1.5) -- (1.7, -0.5);
\draw (1.9,1.5) -- (1.9, -0.5);
\draw (2.1,1.5) -- (2.1, -0.5);
\draw (2.3,1.5) -- (2.3, -0.5);

\begin{scope}[xshift = 6 cm, yshift=-0.825 cm, rotate=90]
\draw (-.75,0) -- (0.5, 0);
\draw (-.75,0.5) -- (-.5,.5);
\draw (-.75,1) -- (-.5,1);
\draw (1,0.5) -- (1.5, 0.5);
\draw (2,1) -- (3.5,1);
\draw (0,1) -- (1.5,1);
\draw (2.5,0.5) -- (2,0.5);
\draw (1,0) -- (2.5,0);
\draw (3, 0.5) -- (3.5, 0.5);
\draw(0,0.5) -- (0.5,0.5);
\draw (3,0) -- (3.5, 0);

\draw (-.5,1) -- (-.35, .85);
\draw (0,.5) -- (-.15, .65);
\draw (-.5,.5) -- (0,1);

\draw (1.5,1) -- (1.65, 0.85);
\draw (2,0.5) -- (1.85, 0.65);
\draw (1.5, 0.5) -- (2,1);

\draw (0.65, 0.15) -- (0.5,0);
\draw (0.85, 0.35) -- (1,.5);
\draw(0.5,0.5) -- (1,0);

\draw (2.5,0)-- (2.65, .15);
\draw (2.85, 0.35) -- (3,0.5);
\draw (2.5,0.5) -- (3, 0);
\end{scope}
\end{tikzpicture}
\end{center}

The diagram on the right above is the simplest braid representation of the figure eight knot. Notice it has three component strands. The braid index of a knot is the least number of strands required to create a knot as the closure of a braid.  Thus the figure eight knot has a braid index of three.

If one of the outer strands of a minimal braid diagram was connected to its neighboring internal strand by only one crossing, the external strand could be removed via Reidemeister relation I (also called stabilization in the braid context), so a braid index greater than three forces the existence of four free crossings. Then seemingly there should be a route to unknotting some remaining crossings to produce a figure eight resultant. Based on our experimental evidence, we propose slightly stronger upper bounding on the unknot probability as this would require 4 nontrivial resultants of a given free knot diagram. 

While the connected sum of foil knots do have braid indices of three or greater, every resultant's crossings are of the same flavor, preventing any figure eight resultants, which have two crossings of each flavor.

\begin{conjecture}
The resultant figure eight probability is less than or equal to the resultant trefoil probability for free knot diagram coming from algebraic tangles.
\end{conjecture}

The free figure eight diagram has 16 resultants: 12 unknots, 2 trefoils, and 2 figure eights. 
So for every free figure eight diagram, an equal number of trefoils and figure eights are created. 
However, trefoils may also be produced without going through a figure eight, 
so the trefoil probability can be greater than that of the figure eight. 

The hypothesis concerning algebraic tangles appears because we know of exactly one counterexample to this conjecture:
The knot $9_{40}$, which is not the closure of an algebraic tangle.  It has 66 resultant trefoils and 78 resultant figure eights. 
Yet this conjecture holds for all other free knots with nine or fewer crossings.

\begin{consequence}
For all algebraic free knot diagrams with n and $n+1$ crossings, where n is odd and $n \geq 3$, the trefoil probability of the $n$-foil is the upper bound on figure eight probability.
\end{consequence}

\begin{proof}
This extends the previous results from trefoils to figure eights using the above conjecture absolutely over all diagrams.
\end{proof}

The largest resultant figure eight percentage for all knots through 8 crossings is shared by $7_7$ and $8_{12}$ at 15.625\%, so there is likely a more restrictive bound to be found.

\begin{conjecture}
The absolute maximum resultant figure eight percentage of 15.625\%.
\end{conjecture}

When moving beyond the algebraic knots, this number appears to remain the upper bound, as $9_{40}$ and $9_{47}$ have respective resultant figure eight percentages of 15.234375\% and 13.28125\%.

\section{The $k$ $n$  knots}\label{nk}

\begin{definition} 
The 2 $n$ knots, the closure of the 2 $n$ tangle, include the trefoil, the figure eight, $5_2, 6_1, 7_2, 8_1, 9_2$, and so on. 
\end{definition}

Unlike the foil knots, the lack of restriction on $n$ holds as any $n$ still results in a knot.  

\begin{center}
\begin{tikzpicture}
\draw (-.25, 2) -- (.25, 2.5);
\draw (.25, 2) -- (-.25, 2.5);

\draw (.25, 1.75) .. controls (.3, 1.875) .. (.25, 2);
\draw (-.25, 1.75) .. controls (-.3, 1.875) .. (-.25, 2);

\draw (-.25, 1.25) -- (.25, 1.75);
\draw (-.25, 1.75) -- (.25, 1.25);

\draw (-.25, 1.25) .. controls (-.35, 1.15) and (-2, .6) .. (-1.42,0);
\draw (.25, 1.25) .. controls (.35, 1.15) and (2, .6) .. (1.42,0);

\draw (-.25, 2.5) .. controls (-1, 3) and (-2.25, .15) .. (-1.42, -.5);
\draw (.25, 2.5) .. controls (1, 3) and (2.25, .15) .. (1.42, -.5);

\begin{scope}[yscale = -1]
\draw (-1.17,0) -- (-1.42,0);
\draw (-1.17,.5) -- (-1.42,.5);
\draw (-1.17,0) -- (-.67, .5);
\draw (-1.17,.5) -- (-.67,0);
\draw (-.67,0)--(-.42,0);
\draw (-.67, .5) -- (-.42,.5);

\node[draw=none] (ellipsis1) at (0.05,0.25) {$\cdots$};

\draw (1.17,0) -- (1.42,0);
\draw (1.17,.5) -- (1.42,.5);
\draw (1.17, 0) -- (.67, .5);
\draw (.67, 0) -- (1.17, .5);
\draw (.67,0) -- (.42,0);
\draw (.67, .5) -- (.42,.5);
\end{scope}

\end{tikzpicture}
\end{center}

These knots have continued fractions $n + \frac{1}{2} = \frac{2n+1}{2}$. Lemma \ref{KnotFracEquiv} then shows $2$ $n$ knots are nontrivial if $n\geq1$, as the unknot is a closure of the $\pm 1$ tangle. Similarly, it confirms 2 $m$ and 2 $n$ knots are distinct knots.

\begin{theorem}
The free $2$ $n$ knots produce $2^{n+1} + 2{n \choose \frac{n-1}{2}}$ unknots and $2{n \choose \frac{n-1}{2}}$ trefoils if $n$ is odd.  If $n$ is even, $2^{n+1} + 2{n \choose \frac{n}{2}}$ unknots are produced.
\end{theorem}

\begin{proof}
One may think about generating these knots by repeatedly twisting a loop (applying Reidemester move 1 in the same region), and then looping the two ends together.  This construction method demonstrates that these knots should produce a lot of unknots! Any time the upper two strands can be separated, every assignment of the remaining $n$ crossings must give an unknot.  This yields $2^n$ different unknot diagrams. Since the vertical twist above has two different trivial assignments, we can reach a trivial tangle two ways, so we get $2^{n+1}$ unknots in this manner.    

To get the second term in the sum, we look at the case where the top two crossings do not cancel via a Reidemeister II move.

Suppose $n$ is odd.  If we choose $k<\frac{n-1}{2}$ crossings to be negative type, we are left with $n-2k>1$ positive type crossings along the bottom, and $2$ nontrivial crossings at the top; in other words, a nontrivial knot.  Similarly, if we choose $k>\frac{n+1}{2}$, the results are nontrivial.  So, consider $k=\frac{n-1}{2}$ and $k=\frac{n+1}{2}$.  
There are four different ways the top two crossings and the bottom crossing appear:

\begin{center}
\begin{tikzpicture}
\foreach \x in {0,2,4,6} {
    \begin{scope}[xshift = \x cm]	
        \draw (0,1) .. controls (-.5, 2.25) .. (0,3.5);
        \draw (0,1.5) .. controls (-.1, 1.75) .. (0,2);
        \draw (0, 2.5) .. controls (-.1, 2.75) .. (0,3);
        \draw (1,1) .. controls (1.5, 2.25) .. (1, 3.5);
        \draw (1, 1.5) .. controls (1.1, 1.75) .. (1,2);
        \draw (1,2.5) .. controls (1.1, 2.75) .. (1, 3);
    \end{scope}};
\foreach \y in {2,3} {
    \begin{scope}[xshift = 0 cm, yshift = \y cm]
        \draw (.25,0) -- (0,0);
        \draw (.25,.5) -- (0,.5);
        \draw (.25,0) -- (.4,.15);
        \draw (.25, .5) -- (.75, 0);
        \draw (.6, .35) -- (.75, .5);
        \draw (.75,0) -- (1,0);
        \draw (.75, .5) -- (1,.5);
    \end{scope}};
\foreach \y in {1} {
    \begin{scope} [xshift = 0 cm, yshift = \y cm]
    	\draw (.25,0) -- (0,0);
    	\draw (.25,.5) -- (0,.5);
    	\draw (.25, 0) -- (.75, .5);
    	\draw (.25, .5) -- (.4, .35);
    	\draw (.6,.15) -- (.75,0);
    	\draw (.75,0)--(1,0);
    	\draw (.75, .5) -- (1,.5);
    \end{scope}};
\foreach \y in {1,2,3} {
    \begin{scope}[xshift = 2 cm, yshift = \y cm]
        \draw (.25,0) -- (0,0);
        \draw (.25,.5) -- (0,.5);
        \draw (.25,0) -- (.4,.15);
        \draw (.25, .5) -- (.75, 0);
        \draw (.6, .35) -- (.75, .5);
        \draw (.75,0) -- (1,0);
        \draw (.75, .5) -- (1,.5);
    \end{scope}};
\foreach \y in {1} {
    \begin{scope}[xshift = 4 cm, yshift = \y cm]
        \draw (.25,0) -- (0,0);
        \draw (.25,.5) -- (0,.5);
        \draw (.25,0) -- (.4,.15);
        \draw (.25, .5) -- (.75, 0);
        \draw (.6, .35) -- (.75, .5);
        \draw (.75,0) -- (1,0);
        \draw (.75, .5) -- (1,.5);
    \end{scope}};
\foreach \y in {2,3} {
    \begin{scope} [xshift = 4 cm, yshift = \y cm]
    	\draw (.25,0) -- (0,0);
    	\draw (.25,.5) -- (0,.5);
    	\draw (.25, 0) -- (.75, .5);
    	\draw (.25, .5) -- (.4, .35);
    	\draw (.6,.15) -- (.75,0);
    	\draw (.75,0)--(1,0);
    	\draw (.75, .5) -- (1,.5);
    \end{scope}};
\foreach \y in {1,2,3} {
    \begin{scope} [xshift = 6 cm, yshift = \y cm]
    	\draw (.25,0) -- (0,0);
    	\draw (.25,.5) -- (0,.5);
    	\draw (.25, 0) -- (.75, .5);
    	\draw (.25, .5) -- (.4, .35);
    	\draw (.6,.15) -- (.75,0);
    	\draw (.75,0)--(1,0);
    	\draw (.75, .5) -- (1,.5);
    \end{scope}};
\end{tikzpicture}
\end{center}
The first and third cases are unknots; the second and fourth, trefoils.

To summarize, we have ${n \choose \frac{n-1}{2}}+{n \choose \frac{n+1}{2}}$ ways to make the free $n$ tangle at the bottom into a $1$ tangle, and for each of these, one nontrivial assignment of the top crossings produces an unknot, and the other nontrivial assignment produces a trefoil.  Observing that ${n \choose \frac{n-1}{2}} = {n \choose \frac{n+1}{2}}$ gives the $2 {n \choose \frac{n-1}{2}}$ term in the sum and for the trefoil.

If $n$ is even, and we choose $k < \frac{n}{2}$ crossings to be negative type, we are left with $n-2k>1$ positive type crossings  along the bottom, and $2$ nontrivial crossings at the top; in other words, a nontrivial knot.  Similarly, if we choose $k>\frac{n}{2}$, the results are nontrivial.  So, consider $k=\frac{n}{2}$.  All the crossings in bottom tangle may be removed by repeated Reidemester II moves.  The resulting knot diagram has only two crossings and hence is trivial.   As there are two ways to make the top crossings nontrivial, there are $2{n \choose \frac{n}{2}}$ additional ways to make diagrams of the unknot.
\end{proof}

We can go further in our understanding of exactly which knots show up, with what frequency, in assignments of crossings to free $2$ $n$ knots.

\begin{theorem}
A free $2$ $n$ knot produces $2{n \choose \frac{n-k}{2}}$ resultant $2$ $k$ knots and $2{n \choose \frac{n-k}{2}}$ resultant $2$ $k-1$ knots, where $n$ and $k$ must either be both even or both odd. \emph{(To calculate the resultant probability for an even/odd resultant with $\ell$ crossings when both $n$ and $k$ are odd/even, let $k=\ell+1$.)}
\end{theorem}

\begin{proof}
Choose $k$ crossings of the connected $n$ tangle, creating $2{n \choose \frac{n-k}{2}}$ possible assignments of an alternating lower tangle for each nontrivial assignment of the upper two crossings. If connecting the upper two crossings creates an alternating knot, the 2 $k$ knot will be formed. However, the other assignment of the 2 crossings will not be alternating. The crossings on either end of the tangle will now both be connected to the same type of the crossing on the higher free 2 tangle.  

\begin{center}
\begin{tikzpicture}
\draw (-.25, 2) -- (.25, 2.5);
\draw (-.25, 2.5) -- (-.1, 2.35);
\draw (.25, 2) -- (.1, 2.15);

\draw (.25, 1.75) .. controls (.3, 1.875) .. (.25, 2);
\draw (-.25, 1.75) .. controls (-.3, 1.875) .. (-.25, 2);

\draw (-.25, 1.75) -- (-.1, 1.6);
\draw (-.25, 1.25) -- (.25, 1.75);
\draw (.25, 1.25) -- (.1, 1.4);

\draw (-.25, 1.25) .. controls (-.35, 1.15) and (-2, .6) .. (-1.42,.5);
\draw (.25, 1.25) .. controls (.35, 1.15) and (2, .6) .. (1.42,.5);

\draw (-.25, 2.5) .. controls (-1, 3) and (-2.25, .15) .. (-1.42, 0);
\draw (.25, 2.5) .. controls (1, 3) and (2.25, .15) .. (1.42, 0);

\draw (-1.17,0) -- (-1.42,0);
\draw (-1.17,.5) -- (-1.42,.5);
\draw (-1.17, 0) -- (-1.02, .15);
\draw (-1.17,.5) -- (-.67,0);
\draw (-.67, .5) -- (-.82, .35);
\draw (-.67,0)--(-.42,0);
\draw (-.67, .5) -- (-.42,.5);

\node[draw=none] (ellipsis1) at (0.05,0.25) {$\cdots$};

\draw (1.17,0) -- (1.42,0);
\draw (1.17,.5) -- (1.42,.5);
\draw (1.17, 0) -- (.67, .5);
\draw (.82,.15) -- (.67,0);
\draw (1.02, .35) -- (1.17, .5);
\draw (.67,0) -- (.42,0);
\draw (.67, .5) -- (.42,.5);

\end{tikzpicture}
\end{center}

Using the following maneuver in reverse will slightly complicate the knot and create a pair of crossings which we can eliminate with a Reidemeister relation II. Recall that the overstrand will be the same on every crossing within the alternating $n$ tangle.

\begin{center}
\begin{tikzpicture}
\begin{scope}
\draw (.25, .-.5) -- (.1, -.65);
\draw (-.25,-.5) -- (.25,-1);
\draw (-.25, -1) -- (-.1, -.85);
\draw (.5,-0.25) -- (.65, -.1);
\draw (1, -.25) -- (.5, .25);
\draw (1, .25) -- (.85, .1);
\draw (-.65, .1) -- (-.5, .25);
\draw (-.5,-0.25) -- (-1, .25);
\draw (-1, -.25) -- (-.85, -.1);

\draw (-.25, -.5) -- (-.5, -.25);
\draw (.25, -.5) -- (.5, -0.25);
\draw (-.5, .25) .. controls (0,.75) .. (.5, .25);
\draw (-1, .45) .. controls (0, 1.4) .. (1,.45);
\draw (-1, .25) .. controls (-1.1, .35) .. (-1, .45);
\draw (1, .25) .. controls (1.1, .35) .. (1, .45);

\draw (-.25,-1) -- (-.425, -1.175);
\draw (.25, -1) -- (.425, -1.175);

\draw (-1,-.25) -- (-1.175, -.425);
\draw (1,-.25) -- (1.175, -.425);

\draw[semithick, dashed] (0, 0) circle (1.25 cm);
\end{scope}

\draw [thick, <->] (1.5,0) -- (2,0);

\begin{scope}[xshift = 3.5cm]
\draw (-.25, -.5) .. controls (0, -.75) .. (-.425, -1.175);
\draw (.25, -.5) .. controls (0, -.75) .. (.425, -1.175);
\draw (.65, .1) -- (.5, .25);
\draw (.5,-0.25) -- (1, .25);
\draw (1, -.25) -- (.85, -.1);
\draw (-.5,-0.25) -- (-.65, -.1);
\draw (-1, -.25) -- (-.5, .25);
\draw (-1, .25) -- (-.85, .1);

\draw (-.25, -.5) -- (-.5, -.25);
\draw (.25, -.5) -- (.5, -0.25);
\draw (-.5, .25) .. controls (0,.75) .. (.5, .25);
\draw (-1, .45) .. controls (0, 1.4) .. (1,.45);
\draw (-1, .25) .. controls (-1.1, .35) .. (-1, .45);
\draw (1, .25) .. controls (1.1, .35) .. (1, .45);

\draw (-1,-.25) -- (-1.175, -.425);
\draw (1,-.25) -- (1.175, -.425);

\draw[semithick, dashed] (0, 0) circle (1.25 cm);
\end{scope}

\draw (5.35,0) node {$\implies$};

\begin{scope}[xshift = 8.25 cm, yshift = -1.25 cm]
\draw (-.25, 2) -- (-.1, 2.15);
\draw (.25, 2) -- (-.25, 2.5);
\draw (.25, 2.5) -- (.1, 2.35);

\draw (.25, 1.75) .. controls (.3, 1.875) .. (.25, 2);
\draw (-.25, 1.75) .. controls (-.3, 1.875) .. (-.25, 2);

\draw (-.25, 1.25) -- (-.1, 1.4);
\draw (-.25, 1.75) -- (.25, 1.25);
\draw (.25, 1.75) -- (.1, 1.6);

\draw (.25, 1.25) .. controls (.35, 1.15) and (2, .6) .. (1.42,.5);
\draw (.25, 2.5) .. controls (1, 3) and (2.25, .15) .. (1.42, 0);

\draw (-1.4, 2.125) -- (-1.25, 1.975);
\draw (-1.4, 1.625) -- (-.9, 2.125);
\draw (-.9, 1.625) -- (-1.05, 1.775);

\draw (-.9, 2.125) .. controls (-.45, 2.6) .. (-.25, 2.5);
\draw (-.9, 1.625) .. controls (-.425, 1.2) .. (-.25, 1.25);

\draw (-1.4, 1.625) .. controls (-1.6, 1.425) and (-2.25, .6) .. (-1.42, .5);
\draw (-1.4, 2.125) .. controls (-2, 2.75) and (-3, .25) .. (-1.42, 0);

\draw (-1.17,0) -- (-1.42,0);
\draw (-1.17,.5) -- (-1.42,.5);
\draw (-1.17, 0) -- (-1.02, .15);
\draw (-1.17,.5) -- (-.67,0);
\draw (-.67, .5) -- (-.82, .35);
\draw (-.67,0)--(-.42,0);
\draw (-.67, .5) -- (-.42,.5);

\node[draw=none] (ellipsis1) at (0.05,0.25) {$\cdots$};

\draw (1.17,0) -- (1.42,0);
\draw (1.17,.5) -- (1.42,.5);
\draw (1.17, 0) -- (.67, .5);
\draw (.82,.15) -- (.67,0);
\draw (1.02, .35) -- (1.17, .5);
\draw (.67,0) -- (.42,0);
\draw (.67, .5) -- (.42,.5);
\end{scope}

\end{tikzpicture}
\end{center}

Upon simplification, one crossing on the tangle will be removed and the 2 $k-1$ knot will be formed. From this, we get $2{n \choose \frac{n-k}{2}}$ of the knot 2 $k$ and $2{n \choose \frac{n-k}{2}}$ of the knot 2 $k-1$. 
\end{proof}

Notice a term in the formula for produced trefoils and unknots, when $n$ is odd, is a special case of this theorem. The trefoil here would be called 2 1 and the unknot 2 0, so each would be produced $2{n \choose \frac{n-1}{2}}$ times from a generic 2 $n$ knot shape when the upper two crossings are not immediately trivializable. 

\begin{corollary}
The expected number of crossings for a resultant of a $2$ $n$ knot is $2^{-(n+1)} \sum^n_{k=2} (2k+3) {n \choose \frac{n-k}{2}}$, where the values of $k$ and $n$ are even, and $2^{-(n+1)} \left(3{n \choose \frac{n-1}{2}} + \sum^n_{k=3} (2k+3) {n \choose \frac{n-k}{2}} \right)$, where the values of $k$ and $n$ are odd.
\end{corollary}

\begin{proof}
For even values of $n$, resultants with $k+2$ and $k+1$ crossings have the same probability of $\frac{2}{2^{n+2}}  {n \choose \frac{n-k}{2}}$. Combining these pairs of resultants into a single term, we get a factor of $2k+3$ for the total number of crossings for a particular duplicated probability. For even $n$, the first pair, the trefoils and the figure eights, has 7 total crossings so the sum will start at $k=2$,  going until $n$, skipping odd numbers already covered by the combined nature of each term. 

For odd values of $n$, resultants with $k+2$ and $k+1$ crossings again share probabilities, but the trefoil is now partnered with the unknot, placing its contribution outside the sum derived for even values of $n$. This sum must first start with the 9 total crossings between the pentafoil and figure eights, so the sum begins at $k=3$, now passing over even numbers.
\end{proof}

\begin{theorem}
The limit of resultant unknot probability as the number of crossings $n$ ($k \leq n$) goes to infinity is 0.5 for 2 n knots.
\end{theorem}

\begin{proof}
The resultant unknot probabilities for a 2 $n$ knot are 

$$  \frac{1}{2} + \frac{1}{2^{n+1}} {n \choose \frac{n}{2}}, \text{	    } \frac{1}{2} + \frac{1}{2^{n+1}} {n \choose \frac{n-1}{2}} $$

\noindent when $n$ is even and odd, respectively. Applying the asymptotic behavior of the binomial coefficient found in Theorem \ref{TrefLimit} for $n=0$ and $n=1$ and noting the same cancellation of the powers of 2 occur here, the second terms in the limit disappear and only the $\frac{1}{2}$ term remains for both odd and even $n$.
\end{proof}

\begin{conjecture}
The free $2$ $n$ knots realize the upper bound on the unknots for minimal free knot diagrams with n+2 crossings.
\end{conjecture}

\begin{consequence}
The absolute maximum value for the unknot percentage from a nontrivial knot is 75\%.
\end{consequence}

\begin{proof}
Similar to Consequence \ref{MaxTref}, we define infinite series $e_n$ and $o_n$ equivalent to the probability of the unknot for the 2 $n$ knots for even and odd $n$.

\begin{equation*}
e_n = \frac{1}{2} + \frac{1}{2^{n+1}} {n \choose \frac{n}{2}}, \text{	    } o_n = \frac{1}{2} + \frac{1}{2^{n+1}} {n \choose \frac{n-1}{2}}
\end{equation*}

\noindent These series are monotonically decreasing if every term is less than equal to the term that preceded it. Again, since we only are using even and odd $n$, each sequence's rank can only increase by 2. Using this, we can see:

\begin{align*}
e_{n+2}-e_n &= \frac{1}{2} + \frac{1}{2^{n+3}} \frac{(n+2)!}{\frac{n+2}{2}! (n+2 - \frac{n+2}{2})!} - \frac{1}{2} - \frac{1}{2^{n+1}} \frac{n!}{\frac{n}{2}! (n- \frac{n}{2})!} \\
&=\frac{1}{2^2}  \frac{n+2}{(\frac{n+2}{2})^2} \left(e_n - \frac{1}{2}\right)- \left(e_n - \frac{1}{2}\right)
\end{align*}

This is less than or equal to 0 when

\begin{equation*}
\frac{1}{2^2}  \frac{n+2}{(\frac{n+2}{2})^2} \leq 1 \text{ or } n \geq -1.
\end{equation*}

Similar analysis of $o_n$ results in the inequality $n^2 +3n+1 \geq 0$. Both of these inequalities are greater than zero for the positive values of $n$ (and $n=0$), so the series are both monotonically decreasing for the unknot and trefoil on. Thus, the maximum probability will come in the beginning, with 100\% for unknot structure 2 0 and then dropping immediately to 75\% for the trefoil, the first nontrivial member of the 2 $n$ family.
\end{proof}

\begin{conjecture}
The free 2 n knots are the lower bound for trefoils for minimal prime algebraic free knot diagrams with n+2 crossings.
\end{conjecture}

As so many assignments of the free 2 $n$ knot are forced to be the unknot, this does not leave much room for trefoils.  If true, the free 2 $n$ knots are not the only knots that realize the lower bound. For example, $7_2$ shares the same resultant trefoil percentage, 15.625\%, with $7_4$ and $7_7$. Again, we insert the algebraic knot qualifier for the lone counterexample of $9_{40}$.

Between the two trefoil bounding conjectures, the space of possible trefoil probabilities would be known. The following plot shows these possible ranges of probabilities for free knots up to 75 crossings in between those for the $n$-foils (blue) and 2 $n$ knots (orange). 

\includegraphics[scale=0.8]{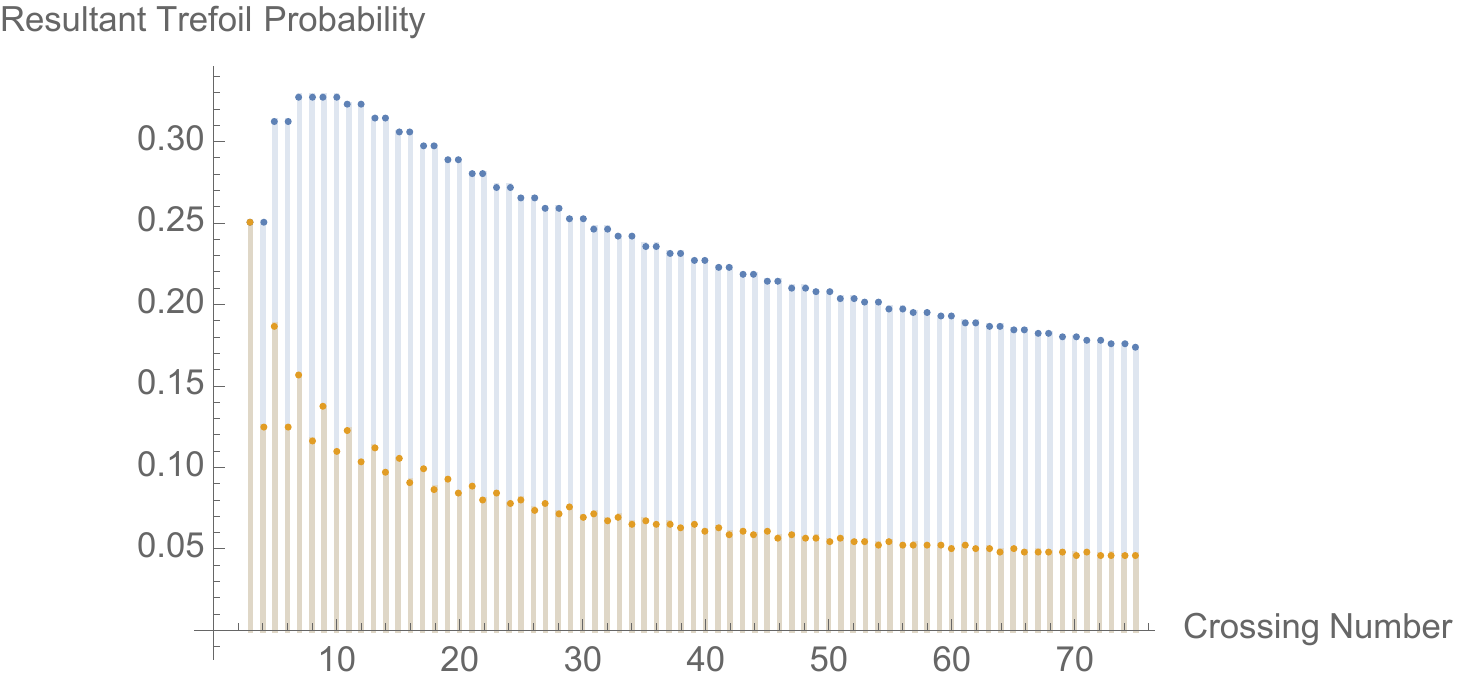}

We can generalize the 2 $n$ knots to create:

\begin{definition}
The $k$ $n$ knots are the generalization of the 2 $n$ knots, and include knots like $7_3$ and $8_3$. As the free closure of the $k$ $n$ tangle, one of $k$ or $n$ must be even. 
\end{definition}

\begin{center}
\begin{tikzpicture}
\begin{scope}
    \draw (.25,0) -- (0,0);
    \draw (.25,.5) -- (0,.5);
    \draw (.25, 0) -- (.75, .5);
    \draw (.25,.5) -- (.75,0);
    \draw (.75,0)--(1,0);
    \draw (.75, .5) -- (1,.5);
    \node[draw=none] (ellipsis1) at (1.425,0.25) {$\cdots$};
    \draw (1.75,0) -- (2,0);
    \draw (1.75,.5) -- (2,.5);
    \draw (2, 0) -- (2.5, .5);
    \draw (2,.5) -- (2.5,0);
    \draw (2.5,0)--(2.75,0);
    \draw (2.5, .5) -- (2.75,.5);
\end{scope}
\begin{scope}[xshift=1.625cm, yshift=1 cm, rotate=90]
    \draw (.25,0) -- (0,0);
    \draw (.25,.5) -- (0,.5);
    \draw (.25, 0) -- (.75, .5);
    \draw (.25,.5) -- (.75,0);
    \draw (.75,0)--(1,0);
    \draw (.75, .5) -- (1,.5);
    \node[draw=none, rotate=90] (ellipsis1) at (1.425,0.25) {$\cdots$};
    \draw (1.75,0) -- (2,0);
    \draw (1.75,.5) -- (2,.5);
    \draw (2, 0) -- (2.5, .5);
    \draw (2,.5) -- (2.5,0);
    \draw (2.5,0)--(2.75,0);
    \draw (2.5, .5) -- (2.75,.5);
\end{scope}
\draw [decorate,decoration={brace,amplitude=10pt},xshift=0pt,yshift=-4pt]  (2.75,0)--(0,0) node [black,midway,yshift=-0.6cm]{$n$};
\draw [decorate,decoration={brace,amplitude=10pt},xshift=0pt,yshift=-4pt]  (1.75,3.75)--(1.75,1.25) node [black,midway,xshift=0.6cm]{$k$};
\draw (0,0) .. controls (-1, 0.3) .. (-.5,2);
\draw (-.5,2) .. controls (0.8, 4.75) .. (1.125, 3.75);
\draw (2.75,0) .. controls (3.75, .3) .. (3.25,2);
\draw (3.25,2) .. controls (1.95,4.75) .. (1.625,3.75);
\draw (0,.5) .. controls (-.4, 0.7) .. (.6, .8);
\draw (.6,.8) .. controls (1.1, 0.9) .. (1.125,1);
\draw (2.75, .5) .. controls (3.15, .7) .. (2.15, .8);
\draw (2.15, .8) .. controls (1.65, .9) .. (1.625,1);
\end{tikzpicture}
\end{center}

Constructing these knots by connecting the ends of a tangle of $k$ crossings by some number $\frac{n}{2}$ Reidemeister II moves makes this even and odd requirement clear. However, in the tangle context, having both $n$ and $k$ being odd would lead to links as we could eliminate all but one crossing from a tangle of our choosing, which would then attach itself to the end of the other tangle, thus creating a foil with even length, or a link. Still, like the 2 $n$ knots, there exists $k$ $n$ knots for all number of crossings. 

\begin{theorem}
The resultant knot probability of a free $k$ $n$ knot is the same as the resultant knot probability for a free $n$ $k$ knot.
\end{theorem}

\begin{proof}
Begin with a free $k$ $n$ knot as depicted above.  Via a spherical isotopy, bring the far right strand around the rest of the free knot diagram so it is now to the left of everything else.  Next isotope the resulting diagram so the $k$ tangle is made horizontal by a 90$^\circ$ clockwise rotation, and the $n$-tangle is made vertical by a 270$^\circ$ clockwise rotation.  The result is the  $n$ $k$ free knot diagram. 
\end{proof}

\begin{theorem}
The number of unknots produced by a  $k$ $n$ free knot diagram, with $k$, $n$ both even is $2^k {n \choose \frac{n}{2}} + 2^n {k \choose \frac{k}{2}} - {n \choose \frac{n}{2}} {k \choose \frac{k}{2}}$.
\end{theorem}

\begin{proof}
Again, like with the 2 $n$ knots, we can reach the unknot by first unraveling and separating one of the component tangles, so that the remaining knot diagram consists of Reidemeister I twists of the unknot. To unravel a tangle, we choose half of its crossings to have positive slope type, and the other half to have negative slope type, and then remove crossings in pairs.   There are ${k \choose \frac{k}{2}}$ ways to make the top $k$ crossings unravelable, and each of these generates an unknot for all $2^n$ assignments of the bottom $n$ crossings, for a total of $2^n{k \choose \frac{k}{2}}$ unknots.

The same process, in the opposite order, leads to another $2^k {n \choose \frac{n}{2}}$ unknots.  However, some unknots are counted twice:  those that are unravelable in both the $k$ tangle and the $n$ tangle.  Thus there are $2^k {n \choose \frac{n}{2}} + 2^n {k \choose \frac{k}{2}} - {n \choose \frac{n}{2}} {k \choose \frac{k}{2}}$ unknots resulting from assigning crossings to the $k$ $n$ free knot diagrams. 
\end{proof}

\begin{theorem}
The number of unknots produced by a $k$ $n$ free knot diagram,  with $k$ odd and $n$ even, is $2^k {n \choose \frac{n}{2}} + 2{k \choose \frac{k-1}{2}} {n \choose \frac{n-2}{2}} $.
\end{theorem}

\begin{proof}
As above, unravelling the $n$ tangle will yield $2^k {n \choose \frac{n}{2}}$ unknots.  

Since $k$ is odd, the $k$ tangle cannot be unravelled.  
However, there are 2${k \choose \frac{k-1}{2}}$ assignments which make the $k$ tangle simplify to a $\pm1$ tangle, resulting in ${k \choose \frac{k-1}{2}}$ free $n$ tangles added to a +1 tangle and ${k \choose \frac{k-1}{2}}$ free $n$ tangles added to a -1 tangle. 
If we then consider how we may get an unknot without unravelling the $n$-tangle (as we already considered this possibility above), we see that we may get unknots from a $+1$ tangle plus a $-2$ tangle, or a $-1$ tangle plus a $+2$ tangle.  
This gives us ${n \choose \frac{n-2}{2}}$ unknots in each case.
\end{proof}

\section{The $2$ $1$ $n$ knots}\label{21n}

\begin{definition}
The 2 1 $n$ knots, are again the closure of the 2 1 $n$ tangle when $n$ is odd. Examples include the figure eight, $6_2$, and $8_2$.
\end{definition}

\begin{center}
\begin{tikzpicture}
\draw (.525, 1.85) -- (1.025, 2.35);
\draw (.625, 2.25) -- (1.025, 1.85);

\draw (1.725, 1.85) -- (2.125, 2.25);
\draw (2.225, 1.85) -- (1.725, 2.35);

\draw (1.125, 1.25) -- (1.625, 1.75);
\draw (1.625, 1.25) -- (1.125,1.75);

\draw (1.625, 1.75) -- (1.725, 1.85);
\draw (1.125, 1.75) -- (1.025, 1.85);

\draw (1.725, 2.35) .. controls (1.3875, 2.65) .. (1.025, 2.35);

\draw (0.625, 2.25) .. controls (0.525, 2.35) .. (0.625, 2.45);
\draw (2.125, 2.25) .. controls (2.225, 2.35) .. (2.125, 2.45);
\draw (0.625, 2.45) .. controls (1.375, 3.1) .. (2.125, 2.45);

\draw (0,0) .. controls (-1, 0.3) .. (.525,1.85);
\draw (2.75,0) .. controls (3.75, .3) .. (2.225, 1.85);

\draw (1.125, 1) -- (1.125,1.25);
\draw (1.625, 1) -- (1.625, 1.25);
\draw (0,.5) .. controls (-.4, 0.7) .. (.6, .8);
\draw (.6,.8) .. controls (1.1, 0.9) .. (1.125,1);
\draw (2.75, .5) .. controls (3.15, .7) .. (2.15, .8);
\draw (2.15, .8) .. controls (1.65, .9) .. (1.625,1);

\draw[dashed] (1.375, 2) circle (1 cm); 

\draw (.25,0) -- (0,0);
\draw (.25,.5) -- (0,.5);
\draw (.25, 0) -- (.75, .5);
\draw (.25,.5) -- (.75,0);
\draw (.75,0)--(1,0);
\draw (.75, .5) -- (1,.5);

\node[draw=none] (ellipsis1) at (1.425,0.25) {$\cdots$};

\draw (1.75,0) -- (2,0);
\draw (1.75,.5) -- (2,.5);
\draw (2, 0) -- (2.5, .5);
\draw (2,.5) -- (2.5,0);
\draw (2.5,0)--(2.75,0);
\draw (2.5, .5) -- (2.75,.5);
\draw [decorate,decoration={brace,amplitude=10pt},xshift=0pt,yshift=-4pt]  (2.75,0)--(0,0) node [black,midway,yshift=-0.6cm]{$n$};\end{tikzpicture}
\end{center}

\begin{theorem}\label{21nunknot}
A free $2$ $1$ $n$ knot produces $12{n \choose \frac{n-1}{2}} + 2{n \choose \frac{n-3}{2}}$ unknots.
\end{theorem}

\begin{proof}
The $2$ $1$ structure has eight possible assignments, four of which (red) simplify to $0$ tangle when a Reidemeister II move can be applied to the upper two crossings.  Then the knot in question reduces to a foil.
Each of these produces $2{n \choose \frac{n-1}{2}}$ unknots, so the four trivial assignments give the first $8{n \choose \frac{n-1}{2}}$ unknots. 

\begin{center}
\begin{tikzpicture}
\draw (-.25,-.5) -- (.25,-1);
\draw (-.25, -1) -- (.25, -.5);
\draw (.5,-0.25) -- (1, .25);
\draw (1, -.25) -- (.5, .25);
\draw (-.5, -.25) -- (-1, .25);
\draw (-.5, .25) -- (-1, -.25);

\draw (-.25, -.5) -- (-.5, -.25);
\draw (.25, -.5) -- (.5, -0.25);
\draw (-.5, .25) .. controls (0,.75) .. (.5, .25);
\draw (-1, .45) .. controls (0, 1.4) .. (1,.45);
\draw (-1, .25) .. controls (-1.1, .35) .. (-1, .45);
\draw (1, .25) .. controls (1.1, .35) .. (1, .45);

\draw (-.25,-1) -- (-.425, -1.175);
\draw (.25, -1) -- (.425, -1.175);

\draw (-1,-.25) -- (-1.175, -.425);
\draw (1,-.25) -- (1.175, -.425);

\draw[semithick, dashed] (0, 0) circle (1.25 cm);

\foreach \x in {25, 67.5, 112.5, 155, -25, -67.5, -112.5, -155} {
\draw[->, thick] (canvas polar cs:angle=\x,radius=1.5cm) -- (canvas polar cs:angle=\x,radius=2.5cm);
\begin{scope} [shift = (\x:4cm)]
\draw (-.25, -.5) -- (-.5, -.25);
\draw (.25, -.5) -- (.5, -0.25);
\draw (-.5, .25) .. controls (0,.75) .. (.5, .25);
\draw (-1, .45) .. controls (0, 1.4) .. (1,.45);
\draw (-1, .25) .. controls (-1.1, .35) .. (-1, .45);
\draw (1, .25) .. controls (1.1, .35) .. (1, .45);
\draw (-.25,-1) -- (-.425, -1.175);
\draw (.25, -1) -- (.425, -1.175);

\draw (-1,-.25) -- (-1.175, -.425);
\draw (1,-.25) -- (1.175, -.425);
\end{scope}
}

\foreach \y in {25, 67.5, 112.5, 155} {
\draw[semithick, dashed, red] (\y:4cm) circle (1.25cm);
}
\foreach \y in {-25, -67.5} {
\draw[semithick, dashed, blue] (\y:4cm) circle (1.25cm);
}
\foreach \y in {-112.5, -155} {
\draw[semithick, dashed, green] (\y:4cm) circle (1.25cm);
}

\foreach \a in {25, 67.5, -112.5, -25} {
\begin{scope}[shift = (\a: 4cm)]
\draw (.5,-0.25) -- (.65, -.1);
\draw (1, -.25) -- (.5, .25);
\draw (1, .25) -- (.85, .1);
\end{scope}
}
\foreach \b in {112.5, 155, -155, -67.5} {
\begin{scope}[shift = (\b: 4cm)]
\draw (.65, .1) -- (.5, .25);
\draw (.5,-0.25) -- (1, .25);
\draw (1, -.25) -- (.85, -.1);
\end{scope}
}
\foreach \c in {112.5, 155, -112.5, -25} {
\begin{scope}[shift = (\c: 4cm)]
\draw (-.65, .1) -- (-.5, .25);
\draw (-.5,-0.25) -- (-1, .25);
\draw (-1, -.25) -- (-.85, -.1);
\end{scope}
}
\foreach \d in {25, 67.5, -155, -67.5} {
\begin{scope}[shift = (\d: 4cm)]
\draw (-.5,-0.25) -- (-.65, -.1);
\draw (-1, -.25) -- (-.5, .25);
\draw (-1, .25) -- (-.85, .1);
\end{scope}
}
\foreach \e in {25, 155, -155, -25} {
\begin{scope}[shift = (\e: 4cm)]
\draw (-.25, .-.5) -- (-.1, -.65);
\draw (-.25,-1) -- (.25, -.5);
\draw (.25, -1) -- (.1, -.85);
\end{scope}
}
\foreach \f in {67.5, 112.5, -112.5, -67.5} {
\begin{scope}[shift = (\f: 4cm)]
\draw (.25, .-.5) -- (.1, -.65);
\draw (-.25,-.5) -- (.25,-1);
\draw (-.25, -1) -- (-.1, -.85);
\end{scope}
}
\end{tikzpicture}
\end{center}

Four assignments remain.  Of these, two (green) can have the $2$ $-1$ tangle simplified to $-2$ (or, vice versa, $-2$ $1$ simplified to $2$).  

\begin{center}
\begin{tikzpicture}
\begin{scope}
\draw (.25, .-.5) -- (.1, -.65);
\draw (-.25,-.5) -- (.25,-1);
\draw (-.25, -1) -- (-.1, -.85);
\draw (.5,-0.25) -- (.65, -.1);
\draw (1, -.25) -- (.5, .25);
\draw (1, .25) -- (.85, .1);
\draw (-.65, .1) -- (-.5, .25);
\draw (-.5,-0.25) -- (-1, .25);
\draw (-1, -.25) -- (-.85, -.1);

\draw (-.25, -.5) -- (-.5, -.25);
\draw (.25, -.5) -- (.5, -0.25);
\draw (-.5, .25) .. controls (0,.75) .. (.5, .25);
\draw (-1, .45) .. controls (0, 1.4) .. (1,.45);
\draw (-1, .25) .. controls (-1.1, .35) .. (-1, .45);
\draw (1, .25) .. controls (1.1, .35) .. (1, .45);

\draw (-.25,-1) -- (-.425, -1.175);
\draw (.25, -1) -- (.425, -1.175);

\draw (-1,-.25) -- (-1.175, -.425);
\draw (1,-.25) -- (1.175, -.425);

\draw[semithick, dashed] (0, 0) circle (1.25 cm);
\end{scope}

\draw [thick, <->] (1.5,0) -- (2.5,0);

\begin{scope}[xshift = 4cm]
\draw (-.25, -.5) .. controls (0, -.75) .. (-.425, -1.175);
\draw (.25, -.5) .. controls (0, -.75) .. (.425, -1.175);
\draw (.65, .1) -- (.5, .25);
\draw (.5,-0.25) -- (1, .25);
\draw (1, -.25) -- (.85, -.1);
\draw (-.5,-0.25) -- (-.65, -.1);
\draw (-1, -.25) -- (-.5, .25);
\draw (-1, .25) -- (-.85, .1);

\draw (-.25, -.5) -- (-.5, -.25);
\draw (.25, -.5) -- (.5, -0.25);
\draw (-.5, .25) .. controls (0,.75) .. (.5, .25);
\draw (-1, .45) .. controls (0, 1.4) .. (1,.45);
\draw (-1, .25) .. controls (-1.1, .35) .. (-1, .45);
\draw (1, .25) .. controls (1.1, .35) .. (1, .45);

\draw (-1,-.25) -- (-1.175, -.425);
\draw (1,-.25) -- (1.175, -.425);

\draw[semithick, dashed] (0, 0) circle (1.25 cm);
\end{scope}

\end{tikzpicture}
\end{center}
The result is now a free $n$ tangle added to a $2$ (or $-2$) tangle.  If the crossings of the free $n$ tangle are assigned so that, after pairwise removal of cancelling crossings, a $-1$ or $-3$ (or $1$ or $3$) tangle remains, then the knot diagram represents the unknot.  Thus we've created $2{n \choose \frac{n-1}{2}}$ + $2{n \choose \frac{n-3}{2}}$ unknots.

The two remaining assignments of the free $2$ $1$ tangle (blue) are immutable without connecting them to the $n$ tangle. 
In assigning the crossings of the free $n$ tangle, ${n \choose \frac{n-1}{2}}$ of the assignments simplify to a $1$ tangle, and ${n \choose \frac{n-1}{2}}$ simplify to a $-1$ tangle.  When the result is $2$ $1$ $-1$ or $-2$ $-1$ $1$, it is easily seen to be the unknot.  This contributes the remaining $2 {n \choose \frac{n-1}{2}}$ unknots to the sum.  
\end{proof}

\begin{theorem}
A free $2$ $1$ $n$ knot produces $8{n \choose \frac{n-3}{2}}+2{n \choose \frac{n-1}{2}}+2{n \choose \frac{n-5}{2}}$ trefoils, and, in general, $8{n \choose \frac{n-k}{2}}+2{n \choose \frac{n-k+2}{2}}+2{n \choose \frac{n-k-2}{2}}$ k-foils, k $\leq$ n.
\end{theorem}

\begin{proof}
The structure of this proof is the same as above.
The first four assignments (red) of the free $2$ $1$ $n$ tangle produce only foils, so we get $4 \cdot 2{n \choose \frac{n-3}{2}}$ trefoils. 
In the green assignments, a $\pm 2$ tangle can be added to a $\mp5$ tangle or a $\pm1$ tangle to produce a trefoil:  There are $2{n \choose \frac{n-5}{2}}$ ways to make the free $n$ tangle into a $\mp5$ tangle, and $2{n \choose \frac{n-1}{2}}$ ways to make the free $n$ tangle into a $\pm1$ tangle.

If one's goal is to get a $k$ tangle, there are $4 \cdot 2{n \choose \frac{n-k}{2}}$ ways to get this from the red assignments, and $2{n \choose \frac{n-k-2}{2}} + 2{n \choose \frac{n-k+2}{2}}$ ways from the green.  

The final two assignments (blue) are the closures of the tangles $-2 \, -1 \, m$ and  $2 \, 1 \, m$, which have continued fractions 
$$\displaystyle m + \frac{1}{\pm1 \pm \frac{1}{2}} = m \pm \frac{2}{3} = \frac{3m\pm2}{3}.$$ We appeal to Lemma \ref{KnotFracEquiv}, as any foil would be K$\left(\frac{k}{1}\right)$, where $k$ is an odd integer, so $q=1$ and $q'=3$. Thus these assignments cannot produce any foils, as the difference of 3 and 1 is even, so the odd $k$ ensures $1 \not\equiv 3 \text{ mod } k$ and $3 \not\equiv 1 \text{ mod } k$. 
\end{proof}

\begin{theorem}
 A free $2$ $1$ $n$ knot produces $2{n \choose \frac{n-1}{2}}$ figure eights, and $2{n \choose \frac{n-k}{2}}$ of $2$ $1$ $k$ knots, k $\leq$ n.
\end{theorem}

\begin{proof}
As in the proof of Theorem \ref{21nunknot}, we focus first on the resolution of the $2$ $1$ tangle.  The red and green assignments are all foils, so any figure eight knot must come from a blue assignment.  

A figure eight knot has tangle notation either $2$ $2$ or the equivalent $2$ $1$ $1$ (or their negation).  As the blue tangles are $2$ $1$ and $-2$ $-1$, we look for ways to assign the crossings of the free $n$ tangle so that it simplifies to a single crossing:  for each, there are ${n \choose \frac{n-1}{2}}$ ways to do so.  

If we counting $2$ $1$ $k$ knots instead of figure eight knots, we can assign the crossings of the free $n$ tangle in $2{n \choose \frac{n-k}{2}}$ different ways which result in a $k$ tangle. 
\end{proof}

These results describe a large majority of the resultants of the $2$ $1$ $n$ knots. For example, we have enumerated 966 of the 1024 resultants of the $2$ $1$ $7$ tangle knot diagram. Note we do not examine what happens if the alternating $2$ $1$ tangles are connected to tangles with more than 1 crossing and the overall structures are not alternating. There are an equal number of these unexamined knots as the $2$ $1$ $k$ knots where $k \geq 3$, or
$
2\sum^n_{k=3} {n \choose \frac{n-k}{2}}.
$

In future work we plan to complete this classification, and show

\begin{conjecture}
A free $2$ $1$ $n$ knot produces $2$ of the $3$ $n-1$ knots, $2n$ of the $3$ $n-3$ knots, and $2 {n \choose \frac{k}{2}}$ of the $2$ $n-k$ knots, for $3 \leq n-k \leq n-4$.
\end{conjecture}

(Note these knots are equal in number to the above sum, as the number of 3 $m$ knots is $2+2n = 2 \left( {n \choose 0} + {n \choose 1} \right)$.)

%% file: CombinatorialRandomKnots.bbl
\newpage

\begin{bibdiv}
\begin{biblist}

\bib{MR2079925}{book}{
   author={Adams, Colin C.},
   title={The knot book},
   note={An elementary introduction to the mathematical theory of knots;
   Revised reprint of the 1994 original},
   publisher={American Mathematical Society, Providence, RI},
   date={2004},
   pages={xiv+307},
   isbn={0-8218-3678-1},
   review={\MR{2079925}},
}
\bib{MR0258014}{article}{
   author={Conway, J. H.},
   title={An enumeration of knots and links, and some of their algebraic
   properties},
   conference={
      title={Computational Problems in Abstract Algebra},
      address={Proc. Conf., Oxford},
      date={1967},
   },
   book={
      publisher={Pergamon, Oxford},
   },
   date={1970},
   pages={329--358},
   review={\MR{0258014}},
}
\bib{MR2740363}{article}{
   author={Diao, Y.},
   author={Ernst, C.},
   author={Hinson, K.},
   author={Ziegler, U.},
   title={The mean squared writhe of alternating random knot diagrams},
   journal={J. Phys. A},
   volume={43},
   date={2010},
   number={49},
   pages={495202, 21},
   issn={1751-8113},
   review={\MR{2740363}},
   doi={10.1088/1751-8113/43/49/495202},
}

\bib{MR1981020}{article}{
   author={Dobay, Akos},
   author={Dubochet, Jacques},
   author={Millett, Kenneth},
   author={Sottas, Pierre-Edouard},
   author={Stasiak, Andrzej},
   title={Scaling behavior of random knots},
   journal={Proc. Natl. Acad. Sci. USA},
   volume={100},
   date={2003},
   number={10},
   pages={5611--5615},
   issn={0027-8424},
   review={\MR{1981020}},
   doi={10.1073/pnas.0330884100},
}

\bib{MR899057}{article}{
   author={Kauffman, Louis H.},
   title={State models and the Jones polynomial},
   journal={Topology},
   volume={26},
   date={1987},
   number={3},
   pages={395--407},
   issn={0040-9383},
   review={\MR{899057}},
   doi={10.1016/0040-9383(87)90009-7},
}
\bib{MR1721925}{article}{
   author={Kauffman, Louis H.},
   title={Virtual knot theory},
   journal={European J. Combin.},
   volume={20},
   date={1999},
   number={7},
   pages={663--690},
   issn={0195-6698},
   review={\MR{1721925}},
   doi={10.1006/eujc.1999.0314},
}

\bib{MR1953344}{article}{
   author={Kauffman, Louis H.},
   author={Lambropoulou, Sofia},
   title={Classifying and applying rational knots and rational tangles},
   conference={
      title={Physical knots: knotting, linking, and folding geometric
      objects in $\mathbb R^3$},
      address={Las Vegas, NV},
      date={2001},
   },
   book={
      series={Contemp. Math.},
      volume={304},
      publisher={Amer. Math. Soc., Providence, RI},
   },
   date={2002},
   pages={223--259},
   review={\MR{1953344}},
   doi={10.1090/conm/304/05197},
}

\bib{MR82104}{article}{
   author={Schubert, Horst},
   title={Knoten mit zwei Br\"{u}cken},
   language={German},
   journal={Math. Z.},
   volume={65},
   date={1956},
   pages={133--170},
   issn={0025-5874},
   review={\MR{82104}},
   doi={10.1007/BF01473875},
}
\bib{MR766964}{article}{
   author={Jones, Vaughan F. R.},
   title={A polynomial invariant for knots via von Neumann algebras},
   journal={Bull. Amer. Math. Soc. (N.S.)},
   volume={12},
   date={1985},
   number={1},
   pages={103--111},
   issn={0273-0979},
   review={\MR{766964}},
   doi={10.1090/S0273-0979-1985-15304-2},
}
\end{biblist}
\end{bibdiv}
